\documentclass[a4paper,12pt]{amsart}
\usepackage[margin=3truecm,top=4truecm,bottom=4truecm]{geometry}

\usepackage{amsthm,amsmath,amssymb,comment,setspace,mleftright,constants,cite}

\allowdisplaybreaks 

\usepackage[pagewise]{lineno}
\usepackage{etoolbox} 
\makeatletter 
\newcommand*\linenomathpatch{\@ifstar{\linenomathpatch@AMS}{\linenomathpatch@}}
\newcommand*\linenomathpatch@[1]{
  \expandafter\pretocmd\csname #1\endcsname {\linenomathWithnumbers}{}{}
  \expandafter\pretocmd\csname #1*\endcsname{\linenomathWithnumbers}{}{}
  \expandafter\apptocmd\csname end#1\endcsname {\endlinenomath}{}{}
  \expandafter\apptocmd\csname end#1*\endcsname{\endlinenomath}{}{}
}
\newcommand*\linenomathpatch@AMS[1]{
  \expandafter\pretocmd\csname #1\endcsname {\linenomathWithnumbersAMS}{}{}
  \expandafter\pretocmd\csname #1*\endcsname{\linenomathWithnumbersAMS}{}{}
  \expandafter\apptocmd\csname end#1\endcsname {\endlinenomath}{}{}
  \expandafter\apptocmd\csname end#1*\endcsname{\endlinenomath}{}{}
}
\let\linenomathWithnumbersAMS\linenomathWithnumbers
\patchcmd\linenomathWithnumbersAMS{\advance\postdisplaypenalty\linenopenalty}{}{}{}
\makeatother 

\linenomathpatch{equation}
\linenomathpatch*{gather}
\linenomathpatch*{multline}
\linenomathpatch*{align}
\linenomathpatch*{alignat}
\linenomathpatch*{flalign}

\makeatletter

\@addtoreset{equation}{section}
\makeatother

\newconstantfamily{a}{symbol=\alpha}
\newconstantfamily{E}{symbol=\mathcal{E}}

\theoremstyle{plain} 
\newtheorem{thm}{Theorem}[section]
\newtheorem{prop}[thm]{Proposition}
\newtheorem{cor}[thm]{Corollary}
\newtheorem{lem}[thm]{Lemma}

\theoremstyle{definition}

\theoremstyle{remark}
\newtheorem{rem}[thm]{Remark}

\newcommand{\N}{\mathbb{N}}
\newcommand{\R}{\mathbb{R}}
\newcommand{\Z}{\mathbb{Z}}

\renewcommand{\P}{\mathbb{P}}
\newcommand{\E}{\mathbb{E}}

\DeclareMathOperator*{\essinf}{ess\,inf}

\newcommand{\1}[1]{\mathbf{1}_{#1}}

\newcommand{\half}{\frac{1}{2}}

\renewcommand{\tilde}{\widetilde}

\renewcommand{\bar}{\overline}

\newcommand{\hyphen}{\textrm{-}}
\newcommand{\as}{\textrm{a.s.}}
\renewcommand{\ae}{\textrm{a.e.}}

\begin{document}
\title[Strict comparison for the Lyapunov exponents]{Strict comparison for the Lyapunov exponents of the simple random walk in random potentials}
\author[N.~KUBOTA]{Naoki KUBOTA}
\address[N. Kubota]	{College of Science and Technology, Nihon University, Chiba 274-8501, Japan.}
\email{kubota.naoki08@nihon-u.ac.jp}
\thanks{The author was supported by JSPS KAKENHI Grant Number JP20K14332.}

\begin{abstract}
We consider the simple random walk in i.i.d.~nonnegative potentials on the $d$-dimensional cubic lattice $\Z^d$ ($d \geq 1$).
In this model, the so-called Lyapunov exponent describes the cost of traveling for the simple random walk in the potential.
The Lyapunov exponent depends on the distribution function of the potential, and the aim of this article is to prove that the Lyapunov exponent is strictly monotone in the distribution function of the potential with the order according to strict dominance.
Furthermore, the comparison for the Lyapunov exponent also provides that for the rate function of this model.
\end{abstract}

\maketitle

\section{Introduction}
The main object of study in this article is the so-called Lyapunov exponent, which measures the cost of traveling for the simple random walk in an i.i.d.~nonnegative potential on the $d$-dimensional cubic lattice $\Z^d$ ($d \geq 1$).
We now focus on the fact that the Lyapunov exponent depends on the distribution function of the potential.
Then, the aim of this article is to show that the Lyapunov exponent is strictly monotone in the distribution function of the potential with the order according to strict dominance.
In addition, since the Lyapunov exponent describes the rate function of the large deviation principle for the simple random walk in a random potential, we can lift the strict monotonicity of the Lyapunov exponent to the rate function.

\subsection{The model}\label{subsect:model}
Let $d \geq 1$ and consider the simple random walk $(S_k)_{k=0}^\infty$ on $\Z^d$.
For $x \in \Z^d$, write $P^x$ for the law of the simple random walk starting at $x$, and $E^x$ for the associated expectation.
Independently of $(S_k)_{k=0}^\infty$, let $\omega=(\omega(x))_{x \in \Z^d}$ be a family of i.i.d.~random variables taking values in $[0,\infty)$, and we call $\omega$ the \emph{potential}.
Denote by $\P$ and $\E$ the law of the potential $\omega$ and the associated expectation, respectively.

For any subset $A$ of $\R^d$, $H(A)$ stands for the hitting time of the simple random walk to $A$, i.e.,
\begin{align*}
	H(A):=\inf\{ k \geq 0:S_k \in A \}.
\end{align*}
When $A=\{y\}$ is a single vertex set, we write $H(y):=H(\{y\})$ for simplicity.
Then, define for $x,y \in \Z^d$,
\begin{align*}
	e(x,y,\omega):=E^x\Biggl[ \exp\Biggl\{ -\sum_{k=0}^{H(y)-1}\omega(S_k) \Biggr\} \1{\{ H(y)<\infty \}} \Biggr],
\end{align*}
with the convention that $e(x,y,\omega):=1$ if $x=y$.
Moreover, let us consider the following two-point functions $a(x,y,\omega)$ and $b(x,y)$ on $\Z^d$:
For $x,y \in \Z^d$,
\begin{align*}
	a(x,y,\omega):=-\log e(x,y,\omega)
\end{align*}
and
\begin{align*}
	b(x,y):=-\log\E[e(x,y,\omega)].
\end{align*}
We call $a(x,y,\omega)$ and $b(x,y)$ the \emph{quenched} and \emph{annealed travel costs} from $x$ to $y$ for the simple random walk, respectively.
The quenched travel cost $a(x,y,\omega)$ can be thought of as measuring the cost of traveling from $x$ to $y$ for the simple random walk in a fixed potential $\omega$.
On the other hand, Fubini's theorem and the independence of the potential imply that if $x \not= y$, then
\begin{align*}
	b(x,y)
	= -\log E^x\Biggl[ \prod_{z \in \Z^d}\E\bigl[ \exp\{ -\ell_z(H(y))\omega(0) \} \bigr] \1{\{ H(y)<\infty \}} \Biggr],
\end{align*}
where for $z \in \Z^d$ and $N \in \N$, $\ell_z(N)$ is the number of visits to $z$ by the simple random walk up to time $N-1$, i.e.,
\begin{align*}
	\ell_z(N):=\#\{ 0 \leq k<N:S_k=z \}.
\end{align*}
Hence, the annealed travel cost $b(x,y)$ is rewritten as the quantity after averaging over the potential, and we can interpret $b(x,y)$ as the cost of both optimizing the potential and transporting the simple random walk from $x$ to $y$.
It is easy from the strong Markov property to see that the above travel costs satisfy the following triangle inequalities:
For any $x,y,z \in \Z^d$,
\begin{align*}
	a(x,z,\omega) \leq a(x,y,\omega)+a(y,z,\omega)
\end{align*}
and
\begin{align*}
	b(x,z) \leq b(x,y)+b(y,z).
\end{align*}
(For more details we refer the reader to \cite[(12) in Section~3]{Flu07} and \cite[Proposition~2]{Zer98a}.)

As seen above, this paper treats the quenched and annealed situations simultaneously.
Therefore, to simplify statements, we always make the following assumption only for the quenched situation:
\begin{itemize}
	\item[\bf (Qu)]
		The potential $\omega$ satisfies $\E[\omega(0)]<\infty$ in $d=1$ (there is no additional assumption at all if $d \geq 2$).
\end{itemize}
Under this assumption, the next proposition exhibits the asymptotic behaviors of the travel costs, which were obtained by Flury~\cite[Theorem~A]{Flu07}, Mourrat~\cite[Theorem~{1.1}]{Mou12} and Zerner~\cite[Proposition~4]{Zer98a}. 

\begin{prop}\label{prop:lyaps}
There exist norms $\alpha(\cdot)$ and $\beta(\cdot)$ on $\R^d$ (which are called the \emph{quenched} and \emph{annealed Lyapunov exponents}, respectively) such that for all $x \in \Z^d$,
\begin{align*}
	\lim_{n \to \infty} \frac{1}{n}a(0,nx,\omega)=\alpha(x) \qquad \text{in probability},
\end{align*}
and
\begin{align*}
	\lim_{n \to \infty} \frac{1}{n}b(0,nx)
	= \inf_{n \in \N} \frac{1}{n} b(0,nx)
	= \beta(x).
\end{align*}
Furthermore, the quenched and annealed Lyapunov exponents have the following bounds:
For $x \in \R^d \setminus \{0\}$,
\begin{align*}
	-\log\E[e^{-\omega(0)}] \leq \frac{\alpha(x)}{\|x\|_1} \leq \log(2d)+\E[\omega(0)]
	\qquad (\text{whenever } \E[\omega(0)]<\infty)
\end{align*}
and
\begin{align*}
	-\log\E[e^{-\omega(0)}] \leq \frac{\beta(x)}{\|x\|_1} \leq \log(2d)-\log\E[e^{-\omega(0)}],
\end{align*}
where $\|\cdot\|_1$ is the $\ell^1$-norm on $\R^d$.
\end{prop}

The Lyapunov exponents play a key role in large deviation principles for the simple random walk in random potentials.
For more details, we consider the quenched and annealed path measures $Q_{n,\omega}^\textrm{qu}$ and $Q_n^\textrm{an}$ defined as follows:
\begin{align*}
	\frac{dQ_{n,\omega}^\mathrm{qu}}{dP^0}
	= \frac{1}{Z_{n,\omega}^\mathrm{qu}} \exp\biggl\{ -\sum_{k=0}^{n-1}\omega(S_k) \biggr\}
\end{align*}
and
\begin{align*}
	\frac{dQ_n^\mathrm{an}}{dP^0}
	= \frac{1}{Z_n^\mathrm{an}} \E\biggl[ \exp\biggl\{ -\sum_{k=0}^{n-1}\omega(S_k) \biggr\} \biggr],
\end{align*}
where $Z_{n,\omega}^\mathrm{qu}$ and $Z_n^\mathrm{an}$ are the corresponding normalizing constants.
In addition, write $\alpha(\lambda,\cdot)$ and $\beta(\lambda,\cdot)$ for the quenched and annealed Lyapunov exponents associated with the potential $\omega+\lambda=(\omega(x)+\lambda)_{x \in \Z^d}$, respectively.
Note that $\alpha(\lambda,x)$ and $\beta(\lambda,x)$ are continuous in $(\lambda,x) \in [0,\infty) \times \R^d$ and concave increasing in $\lambda$ (see \cite[Theorem~A]{Flu07} and \cite[below~(64)]{Zer98a}).
Then, define the functions $I$ and $J$ on $\R^d$ as follows:
For $x \in \R^d$,
\begin{align*}
	I(x):=\sup_{\lambda \geq 0}(\alpha(\lambda,x)-\lambda)
\end{align*}
and
\begin{align*}
	J(x):=\sup_{\lambda \geq 0}(\beta(\lambda,x)-\lambda).
\end{align*}
It is known from \cite[below Theorem~A]{Flu07} and \cite[below~(66)]{Zer98a} that $I$ and $J$ are continuous and convex on their effective domains, which are equal to the closed $\ell^1$-unit ball.
The following proposition states the quenched and annealed large deviation principles for the simple random walk in a random potential, which were obtained by Flury~\cite[Theorem~B]{Flu07} and Mourrat~\cite[Theorem~{1.10}]{Mou12}.

\begin{prop}\label{prop:ldp}
Suppose that $\essinf \omega(0)=0$.
Then, the law of $S_n/n$ obeys the following quenched and annealed large deviation principles with the rate functions $I$ and $J$, respectively:
\begin{itemize}
	\item (Quenched case)
		For $\P \hyphen \ae\,\omega$ and for any Borel set $\Gamma$ in $\R^d$,
		\begin{align*}
			-\inf_{x \in \Gamma^o}I(x)
			&\leq \liminf_{n \to \infty} \frac{1}{n}\log Q_{n,\omega}^\mathrm{qu}(S_n \in n\Gamma)\\
			&\leq \limsup_{n \to \infty} \frac{1}{n}\log Q_{n,\omega}^\mathrm{qu}(S_n \in n\Gamma)
				\leq -\inf_{x \in \bar{\Gamma}}I(x).
		\end{align*}
	\item (Annealed case)
		For any Borel set $\Gamma$ in $\R^d$,
		\begin{align*}
			-\inf_{x \in \Gamma^o}J(x)
			&\leq \liminf_{n \to \infty} \frac{1}{n}\log Q_n^\mathrm{an}(S_n \in n\Gamma)\\
			&\leq \limsup_{n \to \infty} \frac{1}{n}\log Q_n^\mathrm{an}(S_n \in n\Gamma)
				\leq -\inf_{x \in \bar{\Gamma}}J(x).
		\end{align*}
\end{itemize}
Here $\Gamma^o$ and $\bar{\Gamma}$ denote the interior and closure of $\Gamma$, respectively.
\end{prop}

\begin{rem}\label{rem:F_dep}
By definition, the annealed travel cost $b(x,y)$, the Lyapunov exponents $\alpha(\cdot)$ and $\beta(\cdot)$ and the rate functions $I$ and $J$ depend on the distribution function of $\omega(0)$, say $\phi$.
If the specification of the dependence on $\phi$ is necessary, we put a subscript $\phi$ on the above notations: $b(x,y)=b_\phi(x,y)$, $\alpha(x)=\alpha_\phi(x)$, $\beta(x)=\beta_\phi(x)$, $I(x)=I_\phi(x)$ and $J(x)=J_\phi(x)$.
\end{rem}

\subsection{Main results}
As mentioned in Remark~\ref{rem:F_dep}, the Lyapunov exponents and the rate functions depend on the distribution function of $\omega(0)$.
In particular, it immediately follows that if $F$ and $G$ are distribution functions on $[0,\infty)$ and satisfy $F \leq G$, then
\begin{align*}
	\alpha_F \geq \alpha_G,\qquad \beta_F \geq \beta_G
\end{align*}
and
\begin{align*}
	I_F \geq I_G,\qquad J_F \geq J_G.
\end{align*}
This raises the question whether we can obtain ``strict'' inequalities in the above inequalities.

To discuss this problem, we introduce the following order between distribution functions on $[0,\infty)$:
For any two distribution functions $F$ and $G$ on $[0,\infty)$, we say that $F$ \emph{strictly dominates} $G$ if $F \leq G$ but $F \not\equiv G$.
Let us now formulate our main results, which are strict comparisons for the quenched and annealed Lyapunov exponents.

\begin{thm}\label{thm:strict_qlyap}
Suppose that $F$ strictly dominates $G$.
Then, there exists a constant $0<\Cl{qlyap}<\infty$ (which may depend on $d$, $F$ and $G$) such that for all $x \in \R^d \setminus \{0\}$,
\begin{align*}
	\alpha_F(x)-\alpha_G(x) \geq \Cr{qlyap}\|x\|_1.
\end{align*}
\end{thm}

\begin{thm}\label{thm:strict_alyap}
Suppose that $F$ strictly dominates $G$.
For $d=1$, assume additionally that
\begin{align}\label{eq:add_a}
	F(0)<e^{-\beta_G(1)}.
\end{align}
Then, there exists a constant $0<\Cl{alyap}<\infty$ (which may depend on $d$, $F$ and $G$) such that for all $x \in \R^d \setminus \{0\}$,
\begin{align*}
	\beta_F(x)-\beta_G(x) \geq \Cr{alyap}\|x\|_1.
\end{align*}
\end{thm}

Since the rate functions are defined by the Lyapunov exponents, strict comparisons for the rate functions are direct consequences of Theorems~\ref{thm:strict_qlyap} and \ref{thm:strict_alyap}.

\begin{cor}\label{cor:strict_rate}
Under the assumption of Theorem~\ref{thm:strict_qlyap} (resp.~Theorem~\ref{thm:strict_alyap}), we have $I_F(x)>I_G(x)$ (resp.~$J_F(x)>J_G(x)$) for all $x \in \R^d$ with $0<\|x\|_1<1$.
\end{cor}

\begin{rem}
It is clear that for any distribution function $\phi$ on $[0,\infty)$, we have $\alpha_\phi(0)=\beta_\phi(0)=0$ and $I_\phi(0)=J_\phi(0)=0$.
This is the reason why we omit the case $\|x\|_1=0$ in Theorems~\ref{thm:strict_qlyap} and \ref{thm:strict_alyap} and Corollary~\ref{cor:strict_rate}.
Furthermore, since the effective domains of the rate functions are equal to the closed $\ell^1$-unit ball, in the case $\|x\|_1>1$, $I_\phi(x)=J_\phi(x)=\infty$ holds for any distribution function $\phi$ on $[0,\infty)$.
Therefore, we can also omit the case $\|x\|_1>1$ in Corollary~\ref{cor:strict_rate}.
However, we do not know whether Corollary~\ref{cor:strict_rate} is still true in the case $\|x\|_1=1$ for a technical reason (see Lemma~\ref{lem:finiteness} below).
\end{rem}

Let us here comment on earlier works related to the above results.
Zerner~\cite{Zer98a} and Flury~\cite{Flu07} first introduced the quenched and annealed Lyapunov exponents for the simple random walk in random potentials, respectively.
In addition, Mourrat~\cite{Mou12} gave optimal conditions for the existence of the quenched Lyapunov exponent.
As mentioned in Subsection~\ref{subsect:model}, the Lyapunov exponents play an important role in large deviation principles for the simple random walk in random potentials.
Accordingly, the Lyapunov exponents have been investigated from various viewpoints.
Flury~\cite{Flu08} and Zygouras~\cite{Zyg09} proved that the quenched and annealed Lyapunov exponents coincide in $d \geq 4$ and the low disorder regime.
In particular, the low disorder regime enables us to study the behaviors of the quenched and annealed Lyapunov exponents well.
In fact, Wang~\cite{Wan01,Wan02} observed that the quenched and annealed Lyapunov exponents were of the same order in the low disorder regime.
After that, Kosygina et al.~\cite{KosMouZer11} improved Wang's result, and explicitly computed the asymptotic behavior of the quenched and annealed Lyapunov exponents as the potential tends to zero.

The aforementioned results compare the quenched and annealed Lyapunov exponents for a fixed law of the potential.
On the other hand, there are a few results on the comparison between Lyapunov exponents for different laws of the potential.
As a work of this topic, Le~\cite{Le17} considered different laws of the potential simultaneously and proved that in $d \geq 3$, the quenched and annealed Lyapunov exponents are continuous in the law of the potential, i.e., if $F_n$ converges weakly to $F$, then we have for all $x \in \R^d$,
\begin{align*}
	\lim_{n \to \infty} \alpha_{F_n}(x)=\alpha_F(x),\qquad
	\lim_{n \to \infty} \beta_{F_n}(x)=\beta_F(x).
\end{align*}
Le's result naturally raises the question whether $\alpha_{F_n}(x)$ (resp.~$\beta_{F_n}(x)$) coincides with $\alpha_F(x)$ (resp.~$\beta_F(x)$) for all sufficiently large $n$, and this is a motivation for the present article.

Our results are also related to the first passage percolation on $\Z^d$.
In this model, a main object of study is the behavior of the \emph{first passage time} $\tau(x,y)$ from $x$ to $y$ defined as follows:
Assign independently to each edge $e$ of $\Z^d$ a nonnegative random weight $t_e$ with a common distribution function $\phi$.
Then, define
\begin{align}\label{eq:fpt}
	\tau(x,y):=\inf\biggl\{ \sum_{e \in \gamma} t_e:\text{$\gamma$ is a lattice path on $\Z^d$ from $x$ to $y$} \biggr\}.
\end{align}
It is known from \cite[Theorem~{2.18}]{Kes86_book} that under some mild moment condition for the weights, there exists a norm $\mu_\phi(\cdot)$ on $\R^d$ (which is called the \emph{time constant}) such that for all $x \in \Z^d$,
\begin{align*}
	\lim_{n \to \infty} \frac{1}{n}\tau(0,nx)=\mu_\phi(x),\qquad \text{a.s.~and in $L^1$}.
\end{align*}
The first passage time and the time constant correspond to the quenched travel cost and the quenched Lyapunov exponent, respectively.
In the context of the first passage percolation, Marchand~\cite{Mar02} and van~den~Berg--Kesten~\cite{vdBerKes93} studied the strict comparison for the time constant, and obtained the following result:
Assume that $d=2$ and $F(0)<1/2$.
If $F$ is \emph{strictly more variable} than $G$, i.e.,
\begin{align*}
	\int_0^\infty h(t) \,dF(t)<\int_0^\infty h(t) \,dG(t)
\end{align*}
for every convex increasing function $h:\R \to \R$ for which the two integrals converge absolutely, then $\mu_F(\xi_1)<\mu_G(\xi_1)$ holds, where $\xi_1$ is the first coordinate vector.
Note that the strict more variability is a much weaker condition than the strict dominance (see \cite[Section~3]{vdBerKes93}).
We believe that Theorems~\ref{thm:strict_qlyap} and \ref{thm:strict_alyap} are established under the strict more variability.
However, it may be difficult to apply the arguments taken in \cite{Mar02,vdBerKes93} to the quenched and annealed Lyapunov exponents.
This is because in \cite{Mar02,vdBerKes93}, the key to deriving the strict comparison for the time constant is the analysis of ``optimal paths'' for the first passage time (which are lattice paths attaining the infimum on the right side of \eqref{eq:fpt}).
For the quenched and annealed travel costs, we cannot fix such an optimal path since the travel costs are averaged over trajectories of the simple random walk.
Hence, the strict dominance is thought of a reasonable order for the strict comparison between Lyapunov exponents.

Although we consider i.i.d.~potentials and the simple random walk on $\Z^d$ in the present and aforementioned articles, let us finally mention results for models with various changes of our setting.
In \cite{JanNurRA20_arXiv,RASepYil13}, the underlying space is $\Z^d$, but the potential is stationary and ergodic and each step of the random walk is in an arbitrary finite set.
Under such a more general setting, \cite{RASepYil13} studied the quenched large deviation principle and \cite{JanNurRA20_arXiv} constructed the quenched Lyapunov exponent.
On the other hand, \cite[Part~II]{Szn98_book} treats a Brownian motion evolving in a Poissonian potential, which is a continuum version of our model.
In that model, the Lyapunov exponent and the large deviation principle were also studied in both the quenched and annealed situations.
For further related works, see the references given in the aforementioned articles.

\subsection{Organization of the paper}
Let us describe how the present article is organized.
In Section~\ref{sect:pre}, we first introduce a coupling of potentials based on the pseudo-inverse function of the distribution function.
Our next purpose is to observe that the strict dominance for distribution functions causes a definite difference between their pseudo-inverse functions (see Lemma~\ref{lem:pseudo} below).
Throughout the paper, this observation is useful to derive a definite difference between Lyapunov exponents.

Section~\ref{sect:qu_strict} is devoted to the proof of Theorem~\ref{thm:strict_qlyap}, which is the strict inequality for the quenched Lyapunov exponent.
The idea of the proof is as follows:
Assume that $F$ strictly dominates $G$, and let $\omega_F$ and $\omega_G$ be the potentials distributed as $F$ and $G$, respectively.
Then, the observation of Section~\ref{sect:pre} yields that with high probability, there exist a lot of sites whose potentials for $F$ and $G$ are definitely different.
Hence, when we focus on such a typical situation, the simple random walk passes through a lot of sites $z$ with a definite gap between $\omega_F(z)$ and $\omega_G(z)$.
This shows that the quenched travel cost in $\omega_F$ is strictly bigger than that in $\omega_G$, and the strict inequality is inherited to the quenched Lyapunov exponents $\alpha_F$ and $\alpha_G$.

In Section~\ref{sect:an_strict}, we prove Theorem~\ref{thm:strict_alyap}, which is the strict inequality for the annealed Lyapunov exponent.
The idea of the proof is essentially the same as the quenched case.
However, since the annealed travel cost is the quantity after averaging over the potential, it is not sufficient to treat only a typical situation as in the quenched case.
Hence, the main task of this section is to construct an event which is typical for both the potential and the simple random walk and is harmless to the comparison for the annealed Lyapunov exponent.
We need a slightly different construction of such an event in $d=1$ and $d \geq 2$.
Therefore, this section is consist of three subsections:
Subsections~\ref{subsect:harmless} and \ref{subsect:pf_anl_multi} treat the proof of Theorem~\ref{thm:strict_alyap} for $d \geq 2$, and Subsection~\ref{subsect:pf_anl_one} gives the proof of Theorem~\ref{thm:strict_alyap} for $d=1$.

The aim of Section~\ref{sect:rf_strict} is to prove Corollary~\ref{cor:strict_rate}, which is the strict inequality for the quenched and annealed rate functions.
This is a direct consequence of Theorems~\ref{thm:strict_qlyap} and \ref{thm:strict_alyap}.

Section~\ref{sect:one-dim} is devoted to the discussion of comparisons for one-dimensional Lyapunov exponents and rate functions without assumptions (Qu) and \eqref{eq:add_a}.
The main work here is to check that \eqref{eq:add_a} is a necessary and sufficient condition for strict comparison between one-dimensional annealed Lyapunov exponents.
It is clear from Theorem~\ref{thm:strict_alyap} that \eqref{eq:add_a} is a sufficient condition for strict comparison between Lyapunov exponents.
On the contrary, the reason why the lack of \eqref{eq:add_a} causes the coincidence of annealed Lyapunov exponents is as follows:
Assume that $F \leq G$ but \eqref{eq:add_a} fails to hold (i.e., $F(0) \geq e^{-\beta_G(1)}$).
Then, for all large $n$,
\begin{align*}
	b_F(0,n) \geq b_G(0,n) \approx n\beta_G(1) \geq -n\log F(0).
\end{align*}
Roughly speaking, $-n\log F(0)$ is regarded as the cost of adjusting all the potentials for $F$ on the interval $[0,n)$ to zero, and this is one of the worst strategies for the annealed travel cost.
It follows that
\begin{align*}
	-n\log F(0) \geq b_F(0,n) \gtrsim n\beta_G(1) \geq -n\log F(0).
\end{align*}
Therefore, we obtain $\beta_F(1)=\beta_G(1)=-\log F(0)$ by dividing $n$ and letting $n \to \infty$, and $\beta_F(1)$ and $\beta_G(1)$ coincide (see Section~\ref{sect:one-dim} for more details).

We close this section with some general notation.
Write $\|\cdot\|_1$ and $\|\cdot\|_\infty$ for the $\ell^1$ and $\ell^\infty$-norms on $\R^d$.
Throughout the paper, $c$, $c'$ and $C_i$, $i=1,2,\dots$, denote some constants with $0<c,c',C_i<\infty$.

\section{Preliminary}\label{sect:pre}
In this section, we introduce a coupling of potentials.
This is useful to compare Lyapunov exponents for different distribution functions simultaneously.
Independently of $(S_k)_{k=0}^\infty$, let $(U(x))_{x \in \Z^d}$ be a family of independent random variables with the uniform distribution on $(0,1)$.
Then, for a given distribution function $\phi$ on $[0,\infty)$, define
\begin{align*}
	\omega_\phi(x):=\phi^{-1}(U(x)),\qquad x \in \Z^d,
\end{align*}
where $\phi^{-1}$ is the pseudo-inverse function of $\phi$:
\begin{align*}
	\phi^{-1}(s):=\sup\{ t \geq 0:\phi(t)<s \},\qquad s \in (0,1),
\end{align*}
with the convention $\sup\emptyset:=0$.
Note that the potential $\omega_\phi=(\omega_\phi(x))_{x \in \Z^d}$ is a family of i.i.d.~random variables with the common distribution function $\phi$.

The following lemma says that the strict dominance for distribution functions causes a definite difference between their pseudo-inverse functions.

\begin{lem}\label{lem:pseudo}
If $F$ strictly dominates $G$, then the following results hold:
\begin{enumerate}
	\item\label{item:pseudo_H} There exists an $\eta_0=\eta_0(F,G)>0$ and a closed interval
		$\mathcal{H}=\mathcal{H}(\eta_0) \subset (0,1)$ with the Lebesgue measure $|\mathcal{H}| \in (0,1)$
		such that for all $s \in \mathcal{H}$,
		\begin{align*}
			F^{-1}(s)-G^{-1}(s) \geq \eta_0.
		\end{align*}
	\item\label{item:pseudo_0} $F(0)<1$ holds.
\end{enumerate}
\end{lem}
\begin{proof}
Let us first prove part~\eqref{item:pseudo_H}.
Since $F$ strictly dominates $G$, we can find some $t' \geq 0$ such that $G(t')>F(t')$.
Then, set
\begin{align*}
	\epsilon:=\half(G(t')-F(t'))>0.
\end{align*}
The right-continuity of $F$ enables us to take $\eta_0>0$ such that $F(t'+\eta_0) \leq F(t')+\epsilon$.
We now consider the interval
\begin{align*}
	\mathcal{H}:=\biggl[ F(t'+\eta_0)+\frac{1}{3}(G(t')-F(t'+\eta_0)),G(t')	-\frac{1}{3}(G(t')-F(t'+\eta_0)) \biggr].
\end{align*}
Clearly, $\mathcal{H}$ is a closed interval included in $(0,1)$ and $|\mathcal{H}| \in (0,1)$ holds.
Moreover, for any $s \in \mathcal{H}$, we have $F^{-1}(s) \geq t'+\eta_0$ and $G^{-1}(s) \leq t'$.
This implies that for all $s \in \mathcal{H}$,
\begin{align*}
	F^{-1}(s)-G^{-1}(s) \geq t'+\eta_0-t'=\eta_0,
\end{align*}
and the proof of part~\eqref{item:pseudo_H} is complete.

To prove part~\eqref{item:pseudo_0}, assume $F(0)=1$.
Then, $F \equiv 1$ follows.
Since $F \leq G$, we have $F \equiv G \equiv 1$. 
This contradicts $F \not\equiv G$, and part~\eqref{item:pseudo_0} is proved.
\end{proof}

\section{Strict inequality for the quenched Lyapunov exponent}\label{sect:qu_strict}
The aim of this section is to prove Theorem~\ref{thm:strict_qlyap}.
To this end, throughout this section, we fix two distribution functions $F$ and $G$ on $[0,\infty)$ such that $F$ strictly dominates $G$, and let $\eta_0=\eta_0(F,G)$ be the constant appearing in Lemma~\ref{lem:pseudo}-\eqref{item:pseudo_H}.
The idea of the proof of Theorem~\ref{thm:strict_qlyap} is as follows:
Since $F$ strictly dominates $G$, Lemma~\ref{lem:pseudo}-\eqref{item:pseudo_H} implies that for each $z \in \Z^d$, one has $\omega_F(z)>\omega_G(z)$ with positive probability.
Hence, in a typical situation, during a certain time interval, the simple random walk starting at $0$ passes through ``enough'' sites $z$ with $\omega_F(z)>\omega_G(z)$.
It follows that with high probability, the travel cost in $\omega_F$ is strictly bigger than that in $\omega_G$, and this strict comparison is inherited to the quenched Lyapunov exponents $\alpha_F$ and $\alpha_G$.

To carry out the above idea, for each $R \in 2\N$, consider the boxes $\Lambda_R(v):=Rv+[-R/2,R/2)^d$, $v \in \Z^d$, which are called $R$-boxes.
Note that $R$-boxes form a partition of $\Z^d$.
Hence, each $z \in \Z^d$ is contained in precisely one $R$-box, and denote by $[z]_R$ the index $v$ such that $z \in \Lambda_R(v)$.
For a given $M \in \N$, we say that an $R$-box $\Lambda_R(v)$ is \emph{$M$-white} if the following conditions \eqref{item:white1} and \eqref{item:white2} hold:
\begin{enumerate}
	\item\label{item:white1}
		$\omega_F(z) \geq \omega_G(z)+\eta_0$ holds for some $z \in \Lambda_R(v)$.
	\item\label{item:white2}
		$\omega_G(z) \leq M$ holds for all $z \in \Lambda_R(v)$.
\end{enumerate}
The next lemma guarantees that if $R$ and $M$ are large enough, then each $R$-box can be $M$-white with high probability.

\begin{lem}\label{lem:white}
We have
\begin{align*}
	\lim_{R \to \infty} \lim_{M \to \infty} \P(\Lambda_R(0) \text{ is $M$-white})=1.
\end{align*}
\end{lem}
\begin{proof}
Since $\omega_G(z)$'s are finite and the event $\{ \omega_G(z) \leq M \text{ for all } x \in \Lambda_R(0) \}$ is increasing in $M$, we have
\begin{align*}
	\lim_{M \to \infty} \P(\Lambda_R(0) \text{ is $M$-white})
	&= \P(\omega_F(z) \geq \omega_G(z)+\eta_0 \text{ for some } z \in \Lambda_R(0))\\
	&= 1-\{ 1-\P(\omega_F(0) \geq \omega_G(0)+\eta_0) \}^{R^d}.
\end{align*}
Note that Lemma~\ref{lem:pseudo}-\eqref{item:pseudo_H} and the definition of $\omega_F$ and $\omega_G$ imply
\begin{align*}
	\P(\omega_F(0) \geq \omega_G(0)+\eta_0)
	&= \int_{(0,1)} \1{\{ F^{-1}(s)-G^{-1}(s) \geq \eta_0 \}}\,ds\\
	&\geq \int_\mathcal{H} \1{\{ F^{-1}(s)-G^{-1}(s) \geq \eta_0 \}}\,ds
		=|\mathcal{H}| \in (0,1).
\end{align*}
Hence,
\begin{align*}
	\lim_{M \to \infty} \P(\Lambda_R(0) \text{ is $M$-white})
	\geq 1-(1-|\mathcal{H}|)^{R^d},
\end{align*}
and the lemma follows by letting $R \to \infty$.
\end{proof}

Define for $0<\delta<p<1$,
\begin{align}\label{eq:D}
	D(\delta\|p):=\delta\log\frac{\delta}{p}+(1-\delta)\log\frac{1-\delta}{1-p}.
\end{align}
It is clear that for each $\delta \in (0,1)$,
\begin{align*}
	\lim_{p \nearrow 1}D(\delta\| p)=\infty
\end{align*}
Moreover, set for $R \in 2\N$ and $M \in \N$,
\begin{align*}
	p_{R,M}:=\P(\Lambda_R(0) \text{ is $M$-white}).
\end{align*}
Thanks to Lemma~\ref{lem:white}, there exist $R$ and $M$ such that
\begin{align}\label{eq:D-log}
	D(1/2\|p_{R,M})>2\log(2d),
\end{align}
and we fix such $R$ and $M$ throughout this section.
The next proposition ensures that with high probability, the simple random walk starting at $0$ must pass through ``enough'' $M$-white $R$-boxes by reaching a remote point.

\begin{prop}\label{prop:QLA}
There exist constants $\Cl{QLA1}$ and $\Cl{QLA2}$ (which may depend on $d$, $F$, $G$, $\eta_0$, $R$ and $M$) such that for all $N \in \N$,
\begin{align*}
	\P(\mathcal{E}(N)^c) \leq \Cr{QLA1}e^{-\Cr{QLA2}N},
\end{align*}
where $\mathcal{E}(N)$ is the event that for all lattice animals $\mathbb{A}$ on $\Z^d$ containing $0$ with $\#\mathbb{A} \geq N$,
\begin{align*}
	\sum_{v \in \mathbb{A}} \1{\{ \Lambda_R(v) \text{ is $M$-white\}}} \geq \frac{\#\mathbb{A}}{2}.
\end{align*}
\end{prop}
\begin{proof}
The union bound shows that
\begin{align*}
	\P(\mathcal{E}(N)^c)
	\leq \sum_{\ell=N}^\infty \sum_\mathbb{A}
		\P\biggl( \sum_{v \in \mathbb{A}}\1{\{ \Lambda_R(v) \text{ is $M$-white} \}}<\frac{\ell}{2} \biggr),
\end{align*}
where the second sum is taken over all lattice animals $\mathbb{A}$ on $\Z^d$ containing $0$ with $\#\mathbb{A}=\ell$.
Note that $(\1{\{ \Lambda_R(v) \text{ is $M$-white} \}})_{v \in \Z^d}$ is a family of independent Bernoulli random variables with parameter $p_{R,M}$.
Hence, we can use the Chernoff bound to estimate the last probability as follows:
\begin{align*}
	\P\biggl( \sum_{v \in \mathcal{\mathbb{A}}}\1{\{ \Lambda_R(v) \text{ is $M$-white} \}}<\frac{\ell}{2} \biggr)
	\leq e^{-\ell D(1/2\|p_{R,M})}.
\end{align*}
Since $(2d)^{2\ell}$ is a rough upper bound on the number of lattice animals on $\Z^d$, of size $\ell$, containing $0$ (see \cite[Lemma~1]{CoxGanGriKes93}), one has
\begin{align*}
	\P(\mathcal{E}(N)^c)
	&\leq \sum_{\ell=N}^\infty (2d)^{2\ell}e^{-\ell D(1/2\|p_{R,M})}\\
	&= \sum_{\ell=N}^\infty \exp\{ -\ell(D(1/2\|p_{R,M})-2\log(2d)) \}.
\end{align*}
Therefore, the proposition immediately follows by \eqref{eq:D-log}.
\end{proof}

We next observe that $M$-white boxes contribute to the difference between the travel costs in $\omega_F$ and $\omega_G$. 
To do this, set for $x,v \in \Z^d$,
\begin{align*}
	\Delta_{F,G}(x):=\omega_F(x)-\omega_G(x)
\end{align*}
and
\begin{align}\label{eq:ET}
	T_R(v):=\inf\{ k>0: S_k \not\in \Lambda_R(v) \}.
\end{align}
Furthermore, define for $x,y \in \Z^d$,
\begin{align*}
	g(x,y):=E^x\Biggl[ \exp\Biggl\{ -\sum_{k=0}^{T_R([S_0]_R)-1}\Delta_{F,G}(S_k) \Biggr\}
	\exp\Biggl\{ -\sum_{k=0}^{H(y)-1}\omega_G(S_k) \Biggr\} \1{\{ H(y)<\infty \}} \Biggr].
\end{align*}

\begin{prop}\label{prop:gap}
If $\Lambda_R(v)$ is $M$-white and $y \not\in \Lambda_R(v)$, then for all $x \in \Lambda_R(v)$,
\begin{align*}
	g(x,y) \leq \delta_0 \times e(x,y,\omega_G),
\end{align*}
where
\begin{align*}
	\delta_0:=1-(1-e^{-\eta_0})\Bigl( \frac{1}{2de^M} \Bigr)^{2dR} \in (0,1).
\end{align*}
\end{prop}
\begin{proof}
Assume that $\Lambda_R(v)$ is $M$-white and fix $x \in \Lambda_R(v)$ and $y \not\in \Lambda_R(v)$.
Furthermore, let $A$ be the set of all sites $z \in \Z^d$ such that $\Delta_{F,G}(z) \geq \eta_0$.
Note that $P^x$-a.s.~on the event $\{ H(A)<T_R(v) \}$,
\begin{align*}
	\exp\Biggl\{ -\sum_{k=0}^{T_R(v)-1} \Delta_{F,G}(S_k) \Biggr\}
	\leq \exp\{ -\Delta_{F,G}(S_{H(A)}) \}
	\leq e^{-\eta_0}.
\end{align*}
This proves
\begin{align*}
	g(x,y)
	&\leq e^{-\eta_0} \times E^x\Biggl[ \exp\Biggl\{ -\sum_{k=0}^{H(y)-1} \omega_G(S_k) \Biggr\}
		\1{\{ H(y)<\infty,\,H(A)<T_R(v) \}} \Biggr]\\
	&\quad +E^x\Biggl[ \exp\Biggl\{ -\sum_{k=0}^{H(y)-1} \omega_G(S_k) \Biggr\}
		\1{\{ H(y)<\infty,\,H(A) \geq T_R(v) \}} \Biggr]\\
	&= e(x,y,\omega_G)
		-(1-e^{-\eta_0})E^x\Biggl[ \exp\Biggl\{ -\sum_{k=0}^{H(y)-1} \omega_G(S_k) \Biggr\}
		\1{\{ H(y)<\infty,\,H(A)<T_R(v) \}} \Biggr].
\end{align*}
To estimate the last expectation, we consider a shortest lattice path $\gamma$ which starts and ends at the same site $x$ and goes through a site in $A$.
Since $\Lambda_R(v)$ is $M$-white, it is clear that $\gamma$ has at most $2dR$ vertices and each $z \in \gamma$ satisfies that $z \in \Lambda_R(v)$ and $\omega_G(z) \leq M$.
This combined with the Markov property implies that
\begin{align*}
	&E^x\Biggl[ \exp\Biggl\{ -\sum_{k=0}^{H(y)-1} \omega_G(S_k) \Biggr\}
		\1{\{ H(y)<\infty,\,H(A)<T_R(v) \}} \Biggr]\\
	&\geq \exp\Biggl\{ -\sum_{z \in \gamma}\omega_G(z) \Biggr\}
		P^x(S_\cdot \text{ follows } \gamma) \times e(x,y,\omega_G)\\
	&\geq \Bigl( \frac{1}{2de^M} \Bigr)^{2dR} \times e(x,y,\omega_G).
\end{align*}
With these observations, one has
\begin{align*}
	g(x,y)
	&\leq e(x,y,\omega_G)-(1-e^{-\eta_0})\Bigl( \frac{1}{2de^M} \Bigr)^{2dR} \times e(x,y,\omega_G)\\
	&= \delta_0 \times e(x,y,\omega_G),
\end{align*}
and the proof is complete.
\end{proof}

We are now in a position to prove Theorem~\ref{thm:strict_qlyap}.

\begin{proof}[\bf Proof of Theorem~\ref{thm:strict_qlyap}]
We first introduce the entrance times $(\sigma_i)_{i=1}^\infty$ in $M$-white $R$-boxes and the exit times $(\tau_i)_{i=1}^\infty$ from them:
Set $\tau_0:=1$ and define for $j \geq 0$,
\begin{align*}
	\sigma_{j+1}:=\inf\{ k \geq \tau_j: S_k \text{ is in an $M$-white $R$-box} \}
\end{align*}
and
\begin{align*}
	\tau_{j+1}:=\inf\{ k>\sigma_{j+1}: S_k \not\in \Lambda_R([S_{\sigma_{j+1}}]_R) \}.
\end{align*}
Furthermore, for $z \in \Z^d$, let $\mathbb{A}_R(z)$ stand for the lattice animal on $\Z^d$ which is made of the labels $v$ of $R$-boxes $\Lambda_R(v)$ visited by the simple random walk up to but not including time $H(\Lambda_R([z]_R))$, i.e.,
\begin{align}\label{eq:LA}
	\mathbb{A}_R(z):=\bigl\{ [S_k]_R:0 \leq k<H(\Lambda_R([z]_R)) \bigr\}.
\end{align}

Fix $x \in \Z^d \setminus \{0\}$ and let $n$ be a sufficiently large integer.
To shorten notation, write $N:=\lfloor n\|x\|_\infty/R \rfloor$ and $L:=\lceil N/2 \rceil$.
We now restrict ourselves on the event $\mathcal{E}(N)$ (which appears in Proposition~\ref{prop:QLA}).
Then, since $P^0$-a.s., $\mathbb{A}_R(nx)$ is a lattice animal on $\Z^d$ containing $0$ with $\#\mathbb{A}_R(nx) \geq N$,
\begin{align*}
	\sum_{v \in \mathbb{A}_R(nx)} \1{\{ \Lambda_R(v) \text{ is $M$-white\}}}
	\geq \half \#\mathbb{A}_R(nx)
	\geq \frac{N}{2}.
\end{align*}
It follows that $P^0$-a.s.,
\begin{align}\label{eq:st}
	\tau_{i-1} \leq \sigma_i<\tau_i \leq H(\Lambda_R([nx]_R)),\qquad 1 \leq i \leq L.
\end{align}
Then, set for $1 \leq i \leq L$,
\begin{align*}
	f_i:=E^0\Biggl[ \exp\Biggl\{ -\sum_{k=0}^{\tau_i-1}\Delta_{F,G}(S_k) \Biggr\}
	\exp\Biggl\{ -\sum_{k=0}^{H(nx)-1}\omega_G(S_k) \Biggr\} \1{\{ H(nx)<\infty \}}\Biggr].
\end{align*}
Note that $P^0$-a.s.,
\begin{align}\label{eq:q_induction}
	e(0,nx,\omega_F) \leq f_L.
\end{align}
On the other hand, the strong Markov property implies that for each $1 \leq i \leq L$,
\begin{align}\label{eq:f}
	f_i
	= E^0\Biggl[ \exp\Biggl\{ -\sum_{k=0}^{\sigma_i-1}\Delta_{F,G}(S_k) \Biggr\}
		\exp\Biggl\{ -\sum_{k=0}^{\sigma_i-1}\omega_G(S_k) \Biggr\} \times g(S_{\sigma_i},nx) \Biggr],
\end{align}
where $g(\cdot,\cdot)$ is the two-point function on $\Z^d$ introduced above Proposition~\ref{prop:gap}.
The definition of $\sigma_i$'s and \eqref{eq:st} yield that $P^0$-a.s., for all $1 \leq i \leq L$, $\Lambda_R([S_{\sigma_i}]_R)$ is $M$-white and  $nx \not\in \Lambda_R([S_{\sigma_i}]_R)$ holds.
Therefore, we can use Proposition~\ref{prop:gap} to obtain that $P^0$-a.s.,
\begin{align*}
	g(S_{\sigma_i},nx)
	&\leq \delta_0 \times e(S_{\sigma_i},nx,\omega_G),\qquad 1 \leq i \leq L.
\end{align*}
Substituting this into \eqref{eq:f} and using the strong Markov property again, one has for each $1 \leq i \leq L$,
\begin{align*}
	f_i
	&\leq \delta_0 \times E^0\Biggl[ \exp\Biggl\{ -\sum_{k=0}^{\sigma_i-1}\Delta_{F,G}(S_k) \Biggr\}
		\exp\Biggl\{ -\sum_{k=0}^{H(nx)-1}\omega_G(S_k) \Biggr\} \1{\{ H(nx)<\infty \}} \Biggr]\\
	&\leq \delta_0 \times f_{i-1},
\end{align*}
with the convention $f_0:=e(0,nx,\omega_G)$.
This together with \eqref{eq:q_induction} implies
\begin{align*}
	e(0,nx,\omega_F) \leq \delta_0^L \times e(0,nx,\omega_G).
\end{align*}
With these observations, on the event $\mathcal{E}(N)$,
\begin{align*}
	a(0,nx,\omega_F) \geq a(0,nx,\omega_G)+L\log\delta_0^{-1},
\end{align*}
and Proposition~\ref{prop:QLA} shows that
\begin{align*}
	\P\bigl( a(0,nx,\omega_F)<a(0,nx,\omega_G)+L\log\delta_0^{-1} \bigr)
	\leq \P(\mathcal{E}(N)^c)
	\leq \Cr{QLA1}e^{-\Cr{QLA2}N}.
\end{align*}
Hence, the definition of the quenched Lyapunov exponent (see Proposition~\ref{prop:lyaps}) proves that for all $x \in \Z^d \setminus \{0\}$,
\begin{align*}
	\alpha_F(x) \geq \alpha_G(x)+\frac{\log\delta_0^{-1}}{2dR} \|x\|_1.
\end{align*}
Since $\alpha_F(\cdot)$ and $\alpha_G(\cdot)$ are norms on $\R^d$ and the constants $\delta_0$ and $R$ are independent of $x$, we can easily extend the above inequality to all $x \in \R^d \setminus \{0\}$.
This completes the proof.
\end{proof}

\section{Strict inequality for the annealed Lyapunov exponent}\label{sect:an_strict}
This section is devoted to the proof of Theorem~\ref{thm:strict_alyap}.
To this end, throughout this section, we fix two distribution functions $F$ and $G$ such that $F$ strictly dominates $G$.
In the quenched situation, the strict comparison follows from typical potentials caused by the event $\mathcal{E}(N)$ (which appears in Proposition~\ref{prop:QLA}).
However, in the annealed situation, we have to consider the travel cost after averaging over the potential, and it is not enough to focus on typical potentials as in the quenched situation.
The key to overcoming this difficulty is how long the simple random walk stays around sites with high potential.
To see this, in Subsection~\ref{subsect:harmless}, we construct some events harmless to the comparison for the annealed Lyapunov exponent.
Since those harmless events are slightly different in one and more dimensions, the proof of Theorem~\ref{thm:strict_alyap} is divided into two subsections (see Subsections~\ref{subsect:pf_anl_multi} and \ref{subsect:pf_anl_one} for $d \geq 2$ and $d=1$, respectively).

\subsection{Some events harmless to the annealed comparison}\label{subsect:harmless}
Assume $d \geq 2$ in this subsection.
For any $\kappa>0$, we say that an $R$-box $\Lambda_R(v)$ is \emph{$\kappa$-good} if $\Lambda_R(v)$ contains a site $z$ of $\Z^d$ with $\omega_F(z) \geq \kappa$.
Our first objective is to observe that an $R$-box can be $\kappa$-good with high probability if $\kappa$ and $R$ are sufficiently small and large, respectively.

\begin{lem}\label{lem:kappa}
There exists $\kappa>0$ such that $\P(\omega_F(0)<\kappa)<1$ holds.
In addition, we have for all $R \in 2\N$,
\begin{align*}
	q_{\kappa,R}:=\P(\Lambda_R(0) \text{ is $\kappa$-good})=1-\P( \omega_F(0)<\kappa)^{R^d}.
\end{align*}
\end{lem}
\begin{proof}
Since we have assumed that $F$ strictly dominates $G$, Lemma~\ref{lem:pseudo}-\eqref{item:pseudo_0} implies $\P(\omega_F(0)=0)=F(0)<1$.
Hence, $\P(\omega_F(0)<\kappa)<1$ holds for some small $\kappa>0$, and the first assertion follows.
The proof of the second assertion is  straightforward because we are now working on the i.i.d.~setting.
\end{proof}

From now on, we fix such a $\kappa$ and an arbitrary $x \in \Z^d \setminus \{ 0 \}$.
Then, the next proposition guarantees that with high probability, the simple random walk starting at $0$ must go through ``enough'' $\kappa$-good $R$-boxes by reaching a remote point.

\begin{prop}\label{prop:E1}
There exists an $R=R(d,\kappa) \in 2\N$ such that for all large $n$,
\begin{align*}
	\P(\Cr{E1}(R,n)^c) \leq \Bigl( \frac{1}{4d}\E[e^{-\omega_G(0)}] \Bigr)^{n\|x\|_1},
\end{align*}
where $\Cl[E]{E1}(R,n)$ is the event that
\begin{align*}
	\sum_{v \in \mathcal{A}}\1{\{ \Lambda_R(v) \text{ is $\kappa$-good} \}}
	\geq \frac{\#\mathcal{A}}{2}
\end{align*}
holds for all lattice animals $\mathcal{A}$ on $\Z^d$ containing $0$ with $\#\mathcal{A} \geq \lfloor n\|x\|_\infty/R \rfloor$.
\end{prop}
\begin{proof}
Thanks to Lemma~\ref{lem:kappa} and the fact that $(\1{\{ \Lambda_R(v) \text{ is $\kappa$-good} \}})_{v \in \Z^d}$ is a family of Bernoulli random variables with parameter $q_{\kappa,R}$, we can apply the same argument as in the proof of Proposition~\ref{prop:QLA} to obtain that for all large $R \in 2\N$ and $n \in \N$ with $n \geq 2R$,
\begin{align}\label{eq:E1}
	\P(\Cr{E1}(R,n)^c)
	\leq 2\exp\biggl\{ -\frac{n\|x\|_1}{2dR}(D(1/2 \| q_{\kappa,R})-2\log(2d)) \biggr\}.
\end{align}
On the other hand, the definition of $D(1/2 \| q_{\kappa,R})$ (see \eqref{eq:D}) implies that for all $R \in 2\N$,
\begin{align*}
	\frac{1}{R}(D(1/2 \| q_{\kappa,R})-2\log(2d))
	\geq \frac{R^{d-1}}{2}\log\P(\omega_F(0)<\kappa)^{-1}-\frac{3}{R}\log(2d).
\end{align*}
Due to the hypothesis $d \geq 2$ and Lemma~\ref{lem:kappa}, the right side above goes to infinity as $R \to \infty$.
With these observations, we can find an $R \in 2\N$ such that for all $n \in \N$, the right side of \eqref{eq:E1} is smaller than or equal to $\{ \E[e^{-\omega_G(0)}]/(4d) \}^{n\|x\|_1}$, and this is the proposition follows.
\end{proof}

Our second objective is to estimate the number of $\kappa$-good $R$-boxes gone through by the simple random walk starting at $0$.
To do this, for $z \in \Z^d$, let $\mathcal{A}_R(z)$ stand for the lattice animal which is made of the labels $v$ of $R$-boxes $\Lambda_R(v)$ visited by the simple random walk up to but not including time $H(z)$, i.e.,
\begin{align*}
	\mathcal{A}_R(z):=\bigl\{ [S_k]_R:0 \leq k<H(z) \bigr\}.
\end{align*}
In addition, denote by $\mathcal{G}_R$ the set of all sites $v$ of $\Z^d$ such that the $R$-boxes $\Lambda_R(v)$ is $\kappa$-good:
\begin{align*}
	\mathcal{G}_R:=\{ v \in \Z^d:\Lambda_R(v) \text{ is $\kappa$-good} \}.
\end{align*}
The next proposition says that with high probability, there are a few $\kappa$-good $R$-boxes which the simple random walk starting at $0$ goes through many times before reaching a remote point.

\begin{prop}\label{prop:E2}
There exists an $M=M(d,\kappa,R) \in \N$ such that for all large $n$,
\begin{align}\label{eq:E2}
\begin{split}
	&\E \otimes E^0\Biggl[ \exp\Biggl\{ -\sum_{k=0}^{H(nx)-1}\omega_F(S_k) \Biggr\}
		\1{\{ H(nx)<\infty \} \cap \Cr{E2}(M,n)^c} \Biggr]\\
	&\leq 2\Bigl( \frac{1}{4d}\E[e^{-\omega_G(0)}] \Bigr)^{n\|x\|_1},
\end{split}
\end{align}
where $\Cl[E]{E2}(M,n)$ is the event that
\begin{align*}
	\sum_{v \in \mathcal{G}_R}
	\1{\{ \text{$(S_k)_{k=0}^{H(nx)}$ goes through $\Lambda_R(v)$ at least $M$ times}\}}
	\leq \frac{1}{3}\#(\mathcal{A}_R(nx) \cap \mathcal{G}_R).
\end{align*}
\end{prop}
\begin{proof}
We use Proposition~\ref{prop:E1} to obtain that for all large $n$, the left side of \eqref{eq:E2} is not larger than
\begin{align}\label{eq:E1E2}
\begin{split}
	&\Bigl( \frac{1}{4d}\E[e^{-\omega_G(0)}] \Bigr)^{n\|x\|_1}\\
	&+\E \otimes E^0\Biggl[ \exp\Biggl\{ -\sum_{k=0}^{H(nx)-1}\omega_F(S_k) \Biggr\}
		\1{\{ H(nx)<\infty \} \cap \Cr{E1}(R,n) \cap \Cr{E2}(M,n)^c} \Biggr].
\end{split}
\end{align}
Hence, our task is to show that the second term of \eqref{eq:E1E2} is bounded from above by $\{ \E[e^{-\omega_G(0)}]/(4d) \}^{n\|x\|_1}$.
To this end, take $M \in \N$ large enough to have
\begin{align*}
	\biggl\{ 1-(1-e^{-\kappa})\Bigl( \frac{1}{2d} \Bigr)^{2dR} \biggr\}^{M/(12dR)}
	\leq \frac{1}{4d}\E[e^{-\omega_G(0)}].
\end{align*}
Moreover, consider the entrance times $(\sigma_i)_{i=1}^\infty$ in $\kappa$-good $R$-boxes and the exit times $(\tau_i)_{i=1}^\infty$ from them:
Set $\tau_0:=1$ and define for $j \geq 0$,
\begin{align*}
	\sigma_{j+1}:=\inf\{ k \geq \tau_j: S_k \text{ is in a $\kappa$-good $R$-box} \}
\end{align*}
and
\begin{align*}
	\tau_{j+1}:=\inf\{ k>\sigma_{j+1}: S_k \not\in \Lambda_R([S_{\sigma_{j+1}}]_R) \}.
\end{align*}
Let $n$ be a sufficiently large integer.
Since $P^0 \hyphen \as$, $\mathcal{A}_R(nx)$ is a lattice animal on $\Z^d$ containing $0$ with $\#\mathcal{A}_R(nx) \geq \lfloor n\|x\|_\infty/R \rfloor$ and
\begin{align*}
	\#(\mathcal{A}_R(nx) \cap \mathcal{G}_R)=\sum_{v \in \mathcal{A}_R(nx)} \1{\{ \Lambda_R(v) \text{ is $\kappa$-good} \}},
\end{align*}
it follows that $\P \otimes P^0$-a.s.~on the event $\{ H(nx)<\infty \} \cap \Cr{E1}(R,n) \cap \Cr{E2}(M,n)^c$,
\begin{align*}
	&\sum_{v \in \mathcal{G}_R}
		\1{\{ \text{$(S_k)_{k=0}^{H(nx)}$ goes through $\Lambda_R(v)$ at least $M$ times}\}}\\
	&> \frac{1}{3}\#(\mathcal{A}_R(nx) \cap \mathcal{G}_R)
		\geq \frac{1}{6}\#\mathcal{A}_R(nx)
		\geq \frac{1}{6} \biggl\lfloor \frac{n\|x\|_\infty}{R} \biggr\rfloor
		\geq \frac{n\|x\|_1}{12dR}.
\end{align*}
Therefore, setting $N:=M \lceil n\|x\|_1/(12dR) \rceil$, one has $P^0$-a.s.~on the event $\{ H(nx)<\infty \} \cap \Cr{E1}(R,n) \cap \Cr{E2}(M,n)^c$,
\begin{align*}
	\tau_{i-1} \leq \sigma_i<\tau_i \leq H(nx),\qquad 1 \leq i \leq N.
\end{align*}
This tells us that the second term of \eqref{eq:E1E2} is bounded from above by
\begin{align*}
	\E \otimes E^0\Biggl[ \prod_{i=1}^N \exp\Biggl\{ -\sum_{k=\sigma_i}^{\tau_i-1} \omega_F(S_k) \Biggr\}
	\1{\{ \sigma_i<\infty \}} \Biggr].
\end{align*}
On the other hand, the same argument as in the proof of Theorem~\ref{thm:strict_qlyap} implies that $\P \hyphen \as$,
\begin{align*}
	E^0\Biggl[ \prod_{i=1}^N \exp\Biggl\{ -\sum_{k=\sigma_i}^{\tau_i-1} \omega_F(S_k) \Biggr\}
	\1{\{ \sigma_i<\infty \}} \Biggr]
	\leq \Bigl\{ 1-(1-e^{-\kappa})\Bigl( \frac{1}{2d} \Bigr)^{dR} \Bigr\}^N.
\end{align*}
By the choice of $M$, the right side above is smaller than or equal to
\begin{align*}
	\Bigl\{ 1-(1-e^{-\kappa})\Bigl( \frac{1}{2d} \Bigr)^{dR} \Bigr\}^{Mn\|x\|_1/(12dR)}
	\leq \Bigl( \frac{1}{4d}\E[e^{-\omega_G(0)}] \Bigr)^{n\|x\|_1}.
\end{align*}
With these observations, $\{ \E[e^{-\omega_G(0)}]/(4d) \}^{n\|x\|_1}$ is an upper bound on the second term of \eqref{eq:E1E2}, and the proof is complete.
\end{proof}
Although Proposition~\ref{prop:E2} estimates the number of $\kappa$-good $R$-boxes which the simple random walk goes through many times, the next proposition provides an estimate for the number of $\kappa$-good $R$-boxes which the simple random walk goes through.

\begin{prop}\label{prop:E3}
There exists a positive integer $A=A(d,\kappa,R) \geq 6$ such that for all large $n$,
\begin{align*}
	&\E \otimes E^0\Biggl[ \exp\Biggl\{ -\sum_{k=0}^{H(nx)-1}\omega_F(S_k) \Biggr\}
		\1{\{ H(nx)<\infty \} \cap \Cr{E3}(A,n)^c} \Biggr]\\
	&\leq \Bigl( \frac{1}{4d}\E[e^{-\omega_G(0)}] \Bigr)^{n\|x\|_1},
\end{align*}
where $\Cl[E]{E3}(A,n)$ is the event that $(S_k)_{k=0}^{H(nx)}$ goes through $\kappa$-good $R$-boxes at most $A\lfloor n\|x\|_\infty/R \rfloor$ times.
\end{prop}
\begin{proof}
Take $A$ large enough to satisfy that $A \geq 6$ and
\begin{align*}
	\Bigl\{ 1-(1-e^{-\kappa})\Bigl( \frac{1}{2d} \Bigr)^{dR} \Bigr\}^{A/(2dR)}
	\leq \frac{1}{4d}\E[e^{-\omega_G(0)}].
\end{align*}
Then, the same argument as in the proof of Proposition~\ref{prop:E2} is applicable, and we obtain that for all large $n$,
\begin{align*}
	&\E \otimes E^0\Biggl[ \exp\Biggl\{ -\sum_{k=0}^{H(nx)-1} \omega_F(S_k) \Biggr\}
		\1{\{ H(nx)<\infty \} \cap \Cr{E3}(A,n)^c} \Biggr]\\
	&\leq \Bigl\{ 1-(1-e^{-\kappa})\Bigl( \frac{1}{2d} \Bigr)^{dR} \Bigr\}^{A\lfloor n\|x\|_\infty/R \rfloor}
		\leq \Bigl( \frac{1}{4d}\E[e^{-\omega_G(0)}] \Bigr)^{n\|x\|_1}.
\end{align*}
Hence, the proof is complete.
\end{proof}

Our third objective is to observe that with high probability, the simple random walk does not stay in any $\kappa$-good $R$-box for a long time.
To this end, we set $L_B=L_B(d,R):=BR^{2d}$ and $\mathcal{R}_k:=\{ S_j:0 \leq j \leq k \}$ for $B,k \in \N$, and begin by proving the following lemma.

\begin{lem}\label{lem:range}
We have
\begin{align*}
	\lim_{B \to \infty} \max_{z \in \Lambda_R(0)} P^z(\mathcal{R}_{L_B} \subset \Lambda_R(0))=0.
\end{align*}
\end{lem}
\begin{proof}
The argument as in \cite[Lemma~{4.2}]{Kub20} tells us that there exists a constant $c$ such that for all large $k$,
\begin{align*}
	P^0\biggl( \#\mathcal{R}_k<\frac{ck^{1/2}}{\log k} \biggr)
	\leq \exp\biggl\{ -\frac{c\lfloor k^{1/2} \rfloor}{\log k} \biggr\}.
\end{align*}
Hence, if $B$ is large enough to have $B \geq R$ and $cB^{1/2}>(2d+1)\log B$, then
\begin{align*}
	&\max_{z \in \Lambda_R(0)} P^z(\mathcal{R}_{L_B} \subset \Lambda_R(0))\\
	&= \max_{z \in \Lambda_R(0)} P^0(\mathcal{R}_{L_B} \subset \Lambda_R(0)-z)
		\leq P^0(\#\mathcal{R}_{L_B} \leq R^d)\\
	&\leq P^0\biggl( \#\mathcal{R}_{L_B}<\frac{cL_B^{1/2}}{\log L_B} \biggr)
		\leq \exp\biggl\{ -\frac{c\lfloor L_B^{1/2} \rfloor}{\log L_B} \biggr\}.
\end{align*}
The most right side converges to zero as $B \to \infty$, and the lemma follows.
\end{proof}

After the preparation above, the next proposition gives our desired conclusion for the staying times in $\kappa$-good $R$-boxes.

\begin{prop}\label{prop:E4}
There exists $B=B(d,R,A) \in \N$ such that for all large $n$,
\begin{align*}
	\P \otimes P^0(\{ H(nx)<\infty\} \cap \Cr{E4}(B,n)^c)
	\leq 3\Bigl( \frac{1}{4d}\E[e^{-\omega_G(0)}] \Bigr)^{n\|x\|_1},
\end{align*}
where $\Cl[E]{E4}(B,n)$ is the event that
\begin{align*}
	\sum_{\substack{i \geq 1\\ \tau_i \leq H(nx)}}\1{\{ \tau_i-\sigma_i \leq L_B \}}
	\geq (1-A^{-2}) \times \#\{ i \geq 1:\tau_i \leq H(nx) \}.
\end{align*}
\end{prop}
\begin{proof}
Let $n$ be a sufficiently large integer.
For simplicity of notation, write $\delta:=1-A^{-2}$ and
\begin{align*}
	N:=\biggl\lceil \half \biggl\lfloor \frac{n\|x\|_\infty}{R} \biggr\rfloor \biggr\rceil-1
	\geq \frac{n\|x\|_1}{2dR}-2.
\end{align*}
In addition, thanks to Proposition~\ref{prop:E1}, we may restrict our attention to the event $\Cr{E1}(R,n)$. 
Then, the simple random walk goes through $\kappa$-good $R$-boxes at least $N$ times before hitting $nx$, and the union bound shows that for any $a>0$,
\begin{align*}
	&P^0(\{ H(nx)<\infty\} \cap \Cr{E4}(B,n)^c)\\
	&\leq \sum_{\ell=N}^\infty P^0\Biggl( \sum_{i=1}^\ell \1{\{ \tau_i-\sigma_i \leq L_B \}}<\delta\ell
		\text{ and } \sigma_i<\infty,\, 1 \leq i \leq \ell \Biggr)\\
	&\leq \sum_{\ell=N}^\infty e^{a\delta\ell}
		E^0\Biggl[ \prod_{i=1}^\ell \exp\bigl\{ -a\1{\{ \tau_i-\sigma_i \leq L_B \}} \bigr\} \1{\{ \sigma_i<\infty \}} \Biggr].
\end{align*}
To estimate the last expectations, set
\begin{align*}
	r_B:=\min_{z \in \Lambda_R(0)}P^z(T_R(0) \leq L_B),
\end{align*}
where $T_R(0)$ is the exit time from the $R$-box $\Lambda_R(0)$ (see \eqref{eq:ET}).
From Lemma~\ref{lem:range}, $\lim_{B \to \infty} r_B=1$ holds, and hence $D(\delta\| r_B)$ goes to infinity as $B \to \infty$.
This enables us to take $B$ large enough to have $r_B>\delta$, $e^{-D(\delta\| r_B)} \leq 1/2$ and
\begin{align*}
	\exp\biggl\{ -\frac{1}{4dR} D(\delta\| r_B) \biggr\} \leq \frac{1}{4d}\E[e^{-\omega_G(0)}].
\end{align*}
The same argument as in the proof of Proposition~\ref{prop:E2} works to obtain
\begin{align*}
	E^0\Biggl[ \prod_{i=1}^\ell \exp\bigl\{ -a\1{\{ \tau_i-\sigma_i \leq L_B \}} \bigr\} \1{\{ \sigma_i<\infty \}} \Biggr]
	\leq \{ 1-(1-e^{-a})r_B \}^\ell.
\end{align*}
It follows that setting $f(a):=-a\delta-\log\{ 1-(1-e^{-a})r_B \}$, one has for any $a>0$
\begin{align*}
	P^0(\{ H(nx)<\infty\} \cap \Cr{E4}(B,n)^c)
	\leq \sum_{\ell=N}^\infty e^{-\ell f(a)}.
\end{align*}
Note that the function $f(a)$ attains its maximum at the point
\begin{align*}
	a_0:=\log\frac{r_B(1-\delta)}{\delta(1-r_B)}>0,
\end{align*}
and $f(a_0)=D(\delta\| r_B)$ holds.
With these observations, on the event $\Cr{E1}(R,n)$,
\begin{align*}
	P^0(\{ H(nx)<\infty\} \cap \Cr{E4}(B,n)^c)
	\leq \sum_{\ell=N}^\infty e^{-\ell D(\delta\| r_B)}
	= \frac{1}{1-e^{-D(\delta\| r_B)}} e^{-ND(\delta\| r_B)}.
\end{align*}
By the choice of $B$, the most right side is smaller than or equal to
\begin{align*}
	2\exp\biggl\{ -\biggl( \frac{n\|x\|_1}{2dR}-2 \biggr)D(\delta\|r_B) \biggr\}
	\leq 2\biggl( \frac{1}{4d}\E[e^{-\omega_G(0)}] \biggr)^{n\|x\|_1},
\end{align*}
and the proof is straightforward.
\end{proof}

\subsection{Proof of Theorem~\ref{thm:strict_alyap} in the multi-dimensional case}\label{subsect:pf_anl_multi}
The aim of this subsection is to prove Theorem~\ref{thm:strict_alyap} in $d \geq 2$.
We fix $x \in \Z^d \setminus \{0\}$ and summarize the events appearing the previous subsection for the convenience of the reader:
\begin{align*}
	&\Cr{E1}(R,n)=\biggl\{
		\begin{minipage}{9.8truecm}
			$\sum_{v \in \mathcal{A}}\1{\{ \Lambda_R(v) \text{ is $\kappa$-good} \}} \geq \#\mathcal{A}/2$ holds for all lattice\\
			animals $\mathcal{A}$ on $\Z^d$ containing $0$ with $\#\mathcal{A} \geq \lfloor n\|x\|_\infty/R \rfloor$
		\end{minipage}
		\biggr\},\\
	&\Cr{E2}(M,n)=\Biggl\{
		\sum_{v \in \mathcal{G}_R}
		\1{\{ \text{$(S_k)_{k=0}^{H(nx)}$ goes through $\Lambda_R(v)$ at least $M$ times}\}}
		\leq \frac{1}{3}\#(\mathcal{A}_R(nx) \cap \mathcal{G}_R)
		\Biggr\},\\
	&\Cr{E3}(A,n)=\Bigl\{
			\text{$(S_k)_{k=0}^{H(nx)}$ goes through $\kappa$-good $R$-boxes
			at most $A\lfloor n\|x\|_\infty/R \rfloor$ times}
		\Bigr\},\\
	&\Cr{E4}(B,n)=\Biggl\{
		\sum_{\substack{i \geq 1\\ \tau_i \leq H(nx)}}\1{\{ \tau_i-\sigma_i \leq L_B \}}
		\geq (1-A^{-2}) \times \#\{ i \geq 1:\tau_i \leq H(nx) \} \Biggr\},
\end{align*}
where $\kappa$, $R=R(d,\kappa)$, $M=M(d,\kappa,R)$, $A=A(d,\kappa,R) \geq 6$ and $B=B(d,R,A)$ are the constants chosen in Lemma~\ref{lem:kappa}, Propositions~\ref{prop:E1}, \ref{prop:E2}, \ref{prop:E3} and \ref{prop:E4}, respectively.
For simplicity of notation, write
\begin{align*}
	\mathcal{E}'(n):=\Cr{E1}(R,n) \cap \Cr{E2}(M,n) \cap \Cr{E3}(A,n) \cap \Cr{E4}(B,n).
\end{align*}

Our first task is to prove that if $n$ is large enough, then $\P \otimes P^0 \hyphen \as$ on the event $\{ H(nx)<\infty \} \cap \mathcal{E}'(n)$,
\begin{align}\label{eq:l_bound}
	\#\{ z \in \Z^d:1 \leq \ell_z(H(nx)) \leq ML_B \} \geq \frac{n\|x\|_1}{12dR}.
\end{align}
To this end, set
\begin{align*}
	&V_1:=\bigl\{ v \in \mathcal{A}_R(nx) \cap \mathcal{G}_R:
		\text{$(S_k)_{k=0}^{H(nx)}$ goes through $\Lambda_R(v)$ at most $M$ times}
		\bigr\},\\
	&V_2:=\biggl\{ v \in \mathcal{A}_R(nx) \cap \mathcal{G}_R:
		\begin{minipage}{6truecm}
			$(S_k)_{k=0}^{H(nx)}$ exits from $\Lambda_R(v)$ within\\
			time $L_B$ each time it visits $\Lambda_R(v)$
		\end{minipage}
		\biggr\},
\end{align*}
and consider the cardinality
\begin{align*}
	\mathcal{N}(n):=\#(V_1 \cap V_2).
\end{align*}
Note that for $v \in V_1 \cap V_2$, $\Lambda_R(v)$ is $\kappa$-good and contains at least one site $z$ of $\Z^d$ with $1 \leq \ell_z(H(nx)) \leq ML_B$.
Hence, $\mathcal{N}(n)$ is a lower bound on the left side of \eqref{eq:l_bound}.
Therefore, for \eqref{eq:l_bound}, it suffices to prove that $\mathcal{N}(n)$ is bounded from below by $n\|x\|_1/(12dR)$.
$\P \otimes P^0 \hyphen \as$ on the event $\Cr{E1}(R,n) \cap \Cr{E2}(M,n)$,
\begin{align*}
	\#V_1
	&\geq \#(\mathcal{A}_R(nx) \cap \mathcal{G}_R)
		-\sum_{v \in \mathcal{G}_R}\1{\{ \text{$(S_k)_{k=0}^{H(nx)}$ goes through $\Lambda_R(v)$ at least $M$ times}\}}\\
	&\geq \frac{2}{3}\#(\mathcal{A}_R(nx) \cap \mathcal{G}_R)
		\geq \frac{1}{3} \#\mathcal{A}_R(nx).
\end{align*}
Moreover, $\P \otimes P^0 \hyphen \as$ on the event $\Cr{E3}(A,n) \cap \Cr{E4}(B,n)$,
\begin{align*}
	\sum_{\substack{i \geq 1\\ \tau_i \leq H(nx)}}\1{\{ \tau_i-\sigma_i>L_B \}}
	&\leq A^{-2} \times \#\{ i \geq 1:\tau_i \leq H(nx) \}\\
	&\leq \frac{1}{A} \biggl\lfloor \frac{n\|x\|_\infty}{R} \biggr\rfloor,
\end{align*}
which guarantees that there exist at most $\lfloor \lfloor n\|x\|_\infty/R \rfloor/A \rfloor$ $\kappa$-good $R$-boxes $\Lambda_R(v)$ such that $(S_k)_{k=0}^{H(nx)}$ stays in $\Lambda_R(v)$ for more than time $L_B$ when it visits $\Lambda_R(v)$.
Therefore, $\P \otimes P^0 \hyphen \as$ on the event $\Cr{E3}(A,n) \cap \Cr{E4}(B,n)$,
\begin{align*}
	\# V_2^c \leq \frac{1}{A} \biggl\lfloor \frac{n\|x\|_\infty}{R} \biggr\rfloor
	\leq \frac{1}{A} \#\mathcal{A}_R(nx).
\end{align*}
The estimates for $\#V_1$ and $\#V_2^c$ above and $A \geq 6$ implies that if $n$ is large enough, then $\P \otimes P^0 \hyphen \as$ on the event $\{ H(nx)<\infty \} \cap \mathcal{E}'(n)$, 
\begin{align*}
	\mathcal{N}(n)
	\geq \#V_1-\# V_2^c
	\geq \bigg( \frac{1}{3}-\frac{1}{A} \biggr)\#\mathcal{A}_R(nx)
		\geq \frac{n\|x\|_1}{12dR},
\end{align*}
and \eqref{eq:l_bound} follows.

We next prove that there exists $\rho_0 \in (0,1)$ (which is independent of $x$) such that for all large $n$,
\begin{align}\label{eq:rho}
\begin{split}
	&\E \otimes E^0\Biggl[ \exp\Biggl\{ -\sum_{k=0}^{H(nx)-1}\omega_F(S_k) \Biggr\}
		\1{\{ H(nx)<\infty \} \cap \mathcal{E}'(n)} \Biggr]\\
	&\leq \rho_0^{n\|x\|_1}\E[e(0,nx,\omega_G)].
\end{split}
\end{align}
To do this, let $\mathcal{L}(n)$ be the event that \eqref{eq:l_bound} holds.
The first assertion tells us that if $n$ is large enough, then the left side of \eqref{eq:rho} is bounded from above by
\begin{align*}
	&\E \otimes E^0\Biggl[ \exp\Biggl\{ -\sum_{k=0}^{H(nx)-1}\omega_F(S_k) \Biggr\}
		\1{\{ H(nx)<\infty \} \cap \mathcal{L}(n)} \Biggr]\\
	&= E^0\Biggl[ \prod_{\substack{z \in \Z^d\\ \ell_z(H(nx)) \geq 1}} \E[e^{-\ell_z(H(nx))\omega_F(0)}]
		\1{\{ H(nx)<\infty \} \cap \mathcal{L}(n)} \Biggr]\\
	&\leq E^0\Biggl[ \prod_{\substack{z \in \Z^d\\ 1 \leq \ell_z(H(nx)) \leq ML_B}}R_z(nx)
		\times \prod_{\substack{z \in \Z^d\\ \ell_z(H(nx)) \geq 1}}
		\E[e^{-\ell_z(H(nx))\omega_G(0)}] \1{\{ H(nx)<\infty \} \cap \mathcal{L}(n)} \Biggr],
\end{align*}
where for $y,z \in \Z^d$,
\begin{align*}
	R_z(y)
	:= \frac{\E[e^{-\ell_z(H(y))\omega_F(0)}]}{\E[e^{-\ell_z(H(y))\omega_G(0)}]}
	= \frac{\int_0^1e^{-\ell_z(H(y))F^{-1}(s)}\,ds}{\int_0^1e^{-\ell_z(H(y))G^{-1}(s)}\,ds}
	\in [0,1].
\end{align*}
We use Lemma~\ref{lem:pseudo}-\eqref{item:pseudo_H} to estimate the denominator in the definition of $R_z(n)$ as follows: For $z \in \Z^d$ with $1 \leq \ell_z(H(nx)) \leq ML_B$,
\begin{align*}
	\int_0^1e^{-\ell_z(H(nx))G^{-1}(s)}\,ds
	&= \int_0^1e^{-\ell_z(H(nx))F^{-1}(s)} \times e^{\ell_z(H(nx))(F^{-1}(s)-G^{-1}(s))}\,ds\\
	&\geq \int_0^1e^{-\ell_z(H(nx))F^{-1}(s)}\,ds+(e^{\eta_0}-1)\int_\mathcal{H}e^{-\ell_z(H(nx))F^{-1}(s)}\,ds\\
	&\geq \int_0^1e^{-\ell_z(H(nx))F^{-1}(s)}\,ds+a,
\end{align*}
where
\begin{align*}
	a:=|\mathcal{H}|(e^{\eta_0}-1)\exp\Bigl\{ -ML_B\sup_{s \in \mathcal{H}}F^{-1}(s) \Bigr\} \in (0,\infty).
\end{align*}
Since the function $f(t):=t/(t+a)$ is increasing in $t \geq 0$, one has for $z \in \Z^d$ with $1 \leq \ell_z(H(nx)) \leq ML_B$,
\begin{align*}
	R_z(nx)
	\leq f\biggl( \int_0^1e^{-\ell_z(H(nx))F^{-1}(s)}\,ds \biggr)
	\leq f\biggl( \int_0^1e^{-F^{-1}(s)}\,ds \biggr)
	=: \rho \in (0,1).
\end{align*}
Accordingly, the left side of \eqref{eq:rho} is not greater than
\begin{align*}
	&E^0\Biggl[ \rho^{\#\{z \in \Z^d:1 \leq \ell_z(H(nx)) \leq ML_B \}}
		\times \prod_{\substack{z \in \Z^d\\ \ell_z(H(nx)) \geq 1}}
		\E[e^{-\ell_z(H(nx))\omega_G(0)}] \1{\{ H(nx)<\infty \} \cap \mathcal{L}(n)} \Biggr]\\
	&\leq \rho^{n\|x\|_1/(12dR)} \times
		E^0\Biggl[ \prod_{\substack{z \in \Z^d\\ \ell_z(H(nx)) \geq 1}}
		\E[e^{-\ell_z(H(nx))\omega_G(0)}] \1{\{ H(nx)<\infty \}} \Biggr]\\
	&= \bigl( \rho^{1/(12dR)} \bigr)^{n\|x\|_1} \times \E[e(0,nx,\omega_G)],
\end{align*}
and we get \eqref{eq:rho} by taking $\rho_0:=\rho^{1/(12dR)}$.

Let us finally complete the proof of Theorem~\ref{thm:strict_alyap} for $d \geq 2$.
For a given $x \in \Z^d \setminus \{0\}$, Propositions~\ref{prop:E1}, \ref{prop:E2}, \ref{prop:E3} and \ref{prop:E4} and \eqref{eq:rho} imply that for all large $n$,
\begin{align*}
	\E[e(0,nx,\omega_F)]
	\leq 7\Bigl( \frac{1}{4d}\E[e^{-\omega_G(0)}] \Bigr)^{n\|x\|_1}+\rho_0^{n\|x\|_1} \times \E[e(0,nx,\omega_G)].
\end{align*}
Since $\E[e(0,nx,\omega_G)] \geq \{ \E[e^{-\omega_G(0)}]/(2d) \}^{n\|x\|_1}$, we have for all large $n$,
\begin{align*}
	\E[e(0,nx,\omega_F)]
	\leq 2\max\Bigl\{ 7\Bigl( \half \Bigr)^{n\|x\|_1}, \rho_0^{n\|x\|_1}\Bigr\} \times \E[e(0,nx,\omega_G)],
\end{align*}
or equivalently
\begin{align*}
	b_F(0,nx) \geq b_G(0,nx)-\log 2-\log\max\biggl\{ 7\Bigl( \half \Bigr)^{n\|x\|_1}, \rho_0^{n\|x\|_1}\biggr\}.
\end{align*}
Therefore, dividing by $n$ and letting $n \to \infty$ proves that for any $x \in \Z^d \setminus \{0\}$,
\begin{align*}
	\beta_F(x) \geq \beta_G(x)+\|x\|_1 \min\{ \log 2,-\log\rho_0\}.
\end{align*}
Since $\beta_F(\cdot)$ and $\beta_G(\cdot)$ are norms on $\R^d$ and the constant $\rho_0$ is independent of $x$, we can easily extend the above inequality to $x \in \R^d \setminus \{0\}$, and the proof is complete.\qed

\subsection{Proof of Theorem~\ref{thm:strict_alyap} in the one-dimensional case}\label{subsect:pf_anl_one}
Let $d=1$ and assume \eqref{eq:add_a} (i.e., $F(0)<e^{-\beta_G(1)}$).
In this case, the proof of Theorem~\ref{thm:strict_alyap} is simpler than the case $d \geq 2$.
Since $\beta_F(\cdot)$ and $\beta_G(\cdot)$ are norms on $\R$, it suffices to prove that there exists a constant $\Cl{d=1}$ such that
\begin{align}\label{eq:beta}
	\beta_F(1)-\beta_G(1) \geq \Cr{d=1}.
\end{align}
To this end, we first prepare some notation and lemma.
Due to the assumption~\eqref{eq:add_a}, it is possible to take $\delta \in (0,1)$ such that
\begin{align}\label{eq:F}
	F(0)^{1-\delta}<(1-\delta)e^{-\beta_G(1)}.
\end{align}
Then, for $K,n \in \N$, let $\mathcal{L}'(K,n)$ be the event that 
\begin{align*}
	\#\{ z \in \Z: 1 \leq \ell_z(H(n)) \leq K \} \geq \delta n.
\end{align*}
This event plays a role similar to the event $\mathcal{L}(n)$ in the previous subsection, and the next lemma guarantees that the complement of $\mathcal{L}'(K,n)$ is harmless to the one-dimensional annealed comparison.

\begin{lem}\label{lem:L'}
There exists $K=K(\delta) \in \N$ such that for all $n \geq 1$,
\begin{align}\label{eq:L'}
	\E \otimes E^0\Biggl[ \exp\Biggl\{ -\sum_{k=0}^{H(n)-1}\omega_F(S_k) \Biggr\} \1{\mathcal{L}'(K,n)^c} \Biggr]
	\leq (1-\delta)^ne^{-n\beta_G(1)}.
\end{align}
\end{lem}
\begin{proof}
Note that $P^0 \hyphen \as$ on the event $\mathcal{L}'(K,n)^c$, the number of sites $z$ such that $\ell_z(H(n))>K$ is bigger than $(1-\delta)n$.
Hence, the left side of \eqref{eq:L'} is smaller than or equal to
\begin{align*}
	E^0\Biggl[ \prod_{\substack{z \in \Z\\ \ell_z(H(n))>K}} \E[e^{-\ell_z(H(n))\omega_F(0)}] \1{\mathcal{L}'(K,n)^c} \Biggr]
	\leq \E[e^{-K\omega_F(0)}]^{(1-\delta)n}.
\end{align*}
Lebesgue's dominated convergence theorem together with \eqref{eq:F} shows that
\begin{align*}
	\lim_{K \to \infty} \E[e^{-K\omega_F(0)}]^{1-\delta}
	= F(0)^{1-\delta}
	< (1-\delta)e^{-\beta_G(1)}.
\end{align*}
Therefore, if $K$ is large enough, then the left side of \eqref{eq:L'} is bounded from above by $(1-\delta)^ne^{-n\beta_G(1)}$, and the proof is complete.
\end{proof}

We move to the proof of Theorem~\ref{thm:strict_alyap} in $d=1$.
Lemma~\ref{lem:L'} implies that
\begin{align}\label{eq:conclusion}
	\E[e(0,n,\omega_F)]
	\leq (1-\delta)^ne^{-n\beta_G(1)}
		+\E \otimes E^0\Biggl[ \exp\Biggl\{ -\sum_{k=0}^{H(n)-1}\omega_F(S_k) \Biggr\} \1{\mathcal{L}'(K,n)} \Biggr].
\end{align}
To estimate the right side, we follow the argument used to obtain \eqref{eq:rho}.
The second term in the right side of \eqref{eq:conclusion} is smaller than or equal to
\begin{align*}
	E^0\Biggl[ \prod_{\substack{z \in \Z\\1 \leq \ell_z(H(n)) \leq K}} R_z(n)
	\times \prod_{\substack{z \in \Z\\ \ell_z(H(n)) \geq 1}} \E[e^{-\ell_z(H(n))\omega_G(0)}] \1{\mathcal{L}'(K,n)} \Biggr].
\end{align*}
Note that for any $z \in \Z$ with $1 \leq \ell_z(H(n)) \leq K$,
\begin{align*}
	R_z(n) \leq \frac{\int_0^1e^{-F^{-1}(s)}\,ds}{\int_0^1e^{-F^{-1}(s)}\,ds+a}=:\rho \in (0,1),
\end{align*}
where
\begin{align*}
	a:=|\mathcal{H}|(e^{\eta_0}-1)\exp\Bigl\{ -K\sup_{s \in \mathcal{H}}F^{-1}(s) \Bigr\} \in (0,\infty).
\end{align*}
This, combined with the definition of the annealed Lyapunov exponent (see Proposition~\ref{prop:lyaps}), implies that the second term in the right side of \eqref{eq:conclusion} is bounded from above by
\begin{align*}
	\rho^{\delta n} \times
	E^0\Biggl[ \prod_{\substack{z \in \Z\\ \ell_z(H(n)) \geq 1}} \E[e^{-\ell_z(n)\omega_G(0)}] \1{\mathcal{L}'(K,n)} \Biggr]
	\leq \rho^{\delta n} \times \E[e(0,n,\omega_G)]
	\leq \rho^{\delta n} e^{-n\beta_G(1)}.
\end{align*}
Hence, one has
\begin{align*}
	\E[e(0,n,\omega_F)]
	&\leq (1-\delta)^ne^{-n\beta_G(1)}+\rho^{\delta n}e^{-n\beta_G(1)}\\
	&\leq 2 \max\{ (1-\delta)^n,\rho^{\delta n} \}e^{-n\beta_G(1)},
\end{align*}
which proves that
\begin{align*}
	\frac{1}{n}b_F(0,n) \geq \beta_G(1)-\frac{1}{n}\log 2-\log\max\{ 1-\delta,\rho^\delta \}.
\end{align*}
Consequently, \eqref{eq:beta} immediately follows by letting $n \to \infty$.\qed

\section{Strict inequalities for the rate functions}\label{sect:rf_strict}
This section is devoted to the proof of Corollary~\ref{cor:strict_rate}.

\begin{proof}[\bf Proof of Corollary~\ref{cor:strict_rate}]
Let $\phi$ be a distribution function on $[0,\infty)$.
Recall that $\alpha_\phi(\lambda,\cdot)$ and $\beta_\phi(\lambda,\cdot)$ are the quenched and annealed Lyapunov exponents associated with the potential $\omega_\phi+\lambda=(\omega_\phi(x)+\lambda)_{x \in \Z^d}$, respectively.
For each $x \in \R^d$, we introduce the quantities
\begin{align*}
	\lambda_\phi^\mathrm{qu}(x):=\inf\{ \lambda>0:\partial_- \alpha_\phi(\lambda,x) \leq 1 \}
\end{align*}
and
\begin{align*}
	\lambda_\phi^\mathrm{an}(x):=\inf\{ \lambda>0:\partial_-\beta_\phi(\lambda,x) \leq 1 \},
\end{align*}
where $\partial_-\alpha_\phi(\lambda,x)$ and $\partial_-\beta_\phi(\lambda,x)$ are the left derivatives of $\alpha_\phi(\lambda,x)$ and $\beta_\phi(\lambda,x)$ with respect to $\lambda$, respectively.
Clearly, $\lambda_\phi^\mathrm{qu}(x)$ (resp.~$\lambda_\phi^\mathrm{an}(x)$) attains the supremum in the definition of $I_\phi(x)$ (resp.~$J_\phi(x)$).
Then, the following lemma is the key to proving the corollary.

\begin{lem}\label{lem:finiteness}
We have for any $x \in \R^d$ with $0<\|x\|_1<1$,
\begin{align}\label{eq:finite_a}
	\limsup_{\lambda \to \infty} \frac{\alpha_\phi(\lambda,x)}{\lambda}<1
\end{align}
and
\begin{align}\label{eq:finite_b}
	\limsup_{\lambda \to \infty} \frac{\beta_\phi(\lambda,x)}{\lambda}<1.
\end{align}
\end{lem}

The proof of the lemma is postponed until the end of the section, and we shall complete the proof of the corollary.
Fix $x \in \R^d$ with $0<\|x\|_1<1$.
As mentioned above Proposition~\ref{prop:ldp}, $\alpha_G(\lambda,x)$ and $\beta_G(\lambda,x)$ are concave increasing in $\lambda$.
This together with Lemma~\ref{lem:finiteness} implies that there exists $\lambda_0 \in (0,\infty)$ such that
\begin{align*}
	\partial_- \alpha_G(\lambda_0,x) \leq \frac{\alpha_G(\lambda_0,x)}{\lambda_0}<1
\end{align*}
and
\begin{align*}
	\partial_- \beta_G(\lambda_0,x) \leq \frac{\beta_G(\lambda_0,x)}{\lambda_0}<1,
\end{align*}
which proves that $\lambda_G^\mathrm{qu}(x) \vee \lambda_G^\mathrm{an}(x) \leq \lambda_0<\infty$.
Therefore, Theorems~\ref{thm:strict_qlyap} and \ref{thm:strict_alyap} yield that there exist constants $c$ and $c'$ (which depend on $\lambda_G^\mathrm{qu}(x)$ and $\lambda_G^\mathrm{an}(x)$, respectively) such that
\begin{align*}
	I_F(x)-I_G(x)
	\geq \alpha_F(\lambda_G^\mathrm{qu}(x),x)-\alpha_G(\lambda_G^\mathrm{qu}(x),x)
	\geq c\|x\|_1>0
\end{align*}
and
\begin{align*}
	J_F(x)-J_G(x)
	\geq \beta_F(\lambda_G^\mathrm{an}(x),x)-\beta_G(\lambda_G^\mathrm{an}(x),x)
	\geq c'\|x\|_1>0,
\end{align*}
and the corollary follows.
\end{proof}

We close this section with the proof of Lemma~\ref{lem:finiteness}.

\begin{proof}[\bf Proof of Lemma~\ref{lem:finiteness}]
Fix $x \in \R^d$ with $0<\|x\|_1<1$.
If $\E[\omega_\phi(0)]<\infty$ holds, then Proposition~\ref{prop:lyaps} tells us that
\begin{align*}
	\alpha_\phi(\lambda,x) \leq \|x\|_1(\lambda+\log(2d)+\E[\omega_\phi(0)]).
\end{align*}
Since we have assumed the finiteness of $\E[\omega_\phi(0)]$ in the one-dimensional quenched situation (see assumption~(Qu) above Proposition~\ref{prop:lyaps}), \eqref{eq:finite_a} is valid for $d=1$.
Proposition~\ref{prop:lyaps} also implies that
\begin{align*}
	\beta_\phi(\lambda,x) \leq \|x\|_1\bigl( \lambda+\log(2d)-\log\E[e^{-\omega_\phi(0)}] \bigr),
\end{align*}
and \eqref{eq:finite_b} holds for all $d \geq 1$.

It remains to prove \eqref{eq:finite_a} for $d \geq 2$ (because (Qu) does not guarantee the finiteness of $\E[\omega_\phi(0)]$ for $d \geq 2$).
Although the proof is essentially the same as above, we need some more work to carry out it.
Let $M>0$ and consider the independent Bernoulli site percolation $\eta_M$ on $\Z^d$ defined as
\begin{align*}
	\eta_M=(\eta_M(z))_{z \in \Z^d}:=(\1{\{ \omega_\phi(z) \leq M \}})_{z \in \Z^d}.
\end{align*}
Then, \emph{$M$-clusters} of the configuration $\eta_M$ are the connected components of the graph $\{ z \in \Z^d: \eta_M(z)=1 \}$ with the usual adjacency relation on $\Z^d$: $u,v \in \Z^d$ are adjacent if $\|u-v\|_1=1$.
It is well known that there exists $p_c=p_c(d) \in (0,1)$ such that if $\P(\eta_M(0)=1)>p_c$ holds, then $\P$-a.s., we have a unique infinite $M$-cluster, say $\mathcal{C}_{\infty,M}$, with $\P(0 \in \mathcal{C}_{\infty,M})>0$ (see \cite[Theorems~{1.10} and 8.1]{Gri99_book} for instance).
In addition, define the \emph{chemical distance} $d_M(u,v)$ between $u$ and $v$ as the minimal length of a lattice path from $u$ to $v$ which uses only sites $z$ with $\eta_M(z)=1$.
Note that the chemical distance $d_M(u,v)$ may be equal to infinity if $u$ or $v$ is not in $\mathcal{C}_{\infty,M}$.
To avoid this, for each $z \in \Z^d$, let us consider the closest point to $z$ in $\mathcal{C}_{\infty,M}$ for the $\ell^1$-norm, with a deterministic rule to break ties, and denote it by $\tilde{z}^M$.
From \cite[Lemma~{4.1}]{GarMar10}, there exists a norm $\mu_M(\cdot)$ on $\R^d$ such that for each $y \in \Z^d$,
\begin{align}\label{eq:GM}
	\lim_{n \to \infty} \frac{1}{n}d_M(\tilde{0}^M,\tilde{ny}^M)=\mu_M(y),\qquad \text{$\P$-a.s.~and in $L^1(\P)$}.
\end{align}
Moreover, since
\begin{align*}
	\lim_{M \to \infty}\P(\eta_M(0)=1)=\lim_{M \to \infty} \phi(M)=1,
\end{align*}
we can apply \cite[Theorem~{8.8}]{Gri99_book} and \cite[Corollary~{1.5}]{GarMar07} (or \cite[Theorem~{1.2}]{GarMarProThe17}) to obtain that
\begin{align*}
	\lim_{M \to \infty} \P(0 \in \mathcal{C}_{\infty,M})=1
\end{align*}
and
\begin{align*}
	\lim_{M \to \infty} \sup_{\|x\|_1 \leq 1}|\mu_M(x)-\|x\|_1|=0.
\end{align*}

Fix $x \in \R^d$ with $0<\|x\|_1<1$ and take $M>0$ large enough to have $\P(0 \in \mathcal{C}_{\infty,M})>1/2$ and $\mu_M(x)<1$.
Note that for all $\lambda>0$, $y \in \Z^d$ and $n \in \N$,
\begin{align*}
	0<1-2\P(0 \not\in \mathcal{C}_{\infty,M})
	&\leq \P(0,ny \in \mathcal{C}_{\infty,M})\\
	&\leq \P\bigl( a(0,ny,\omega_\phi+\lambda) \leq d_M(\tilde{0}^M,\tilde{ny}^M)(\lambda+\log(2d)+M) \bigr).
\end{align*}
Hence, Proposition~\ref{prop:lyaps} and \eqref{eq:GM} imply that for all $\lambda>0$ and $y \in \Z^d$,
\begin{align*}
	\frac{\alpha_\phi(\lambda,y)}{\lambda} \leq \frac{1}{\lambda}\mu_M(y)(\lambda+\log(2d)+M).
\end{align*}
This inequality is also valid for all $y \in \R^d$ because $\alpha_\phi(\lambda,\cdot)$ and $\mu_M(\cdot)$ are norms on $\R^d$.
It follows that
\begin{align*}
	\limsup_{\lambda \to \infty} \frac{\alpha_\phi(\lambda,x)}{\lambda} \leq \mu_M(x)<1,
\end{align*}
and \eqref{eq:finite_a} is also valid for $d \geq 2$.
\end{proof}

\section{Discussion on the one-dimensional annealed situation}\label{sect:one-dim}
In the statements of Theorems~\ref{thm:strict_qlyap} and \ref{thm:strict_alyap}, additional conditions (Qu) and \eqref{eq:add_a} are assumed for $d=1$.
Hence, we finally discuss comparisons for one-dimensional Lyapunov exponents and rate functions without (Qu) and \eqref{eq:add_a}.

Let us first comment on the one-dimensional quenched situation without (Qu) (i.e., $\E[\omega(0)]=\infty$).
In this situation, for all $x \in \Z \setminus \{0\}$,
\begin{align*}
	\lim_{n \to \infty} \frac{1}{n}a(0,nx,\omega)=\infty \qquad  \P \hyphen \as
\end{align*}
Indeed, we have for any $L>0$,
\begin{align*}
	a(0,nx,\omega) \geq \sum_{k=0}^{nx-1}(\omega(k) \wedge L),
\end{align*}
and the law of large numbers yields that
\begin{align*}
	\lim_{n \to \infty} \frac{1}{n}\sum_{k=0}^{nx-1}(\omega(k) \wedge L)
	=|x| \times \E[\omega(0) \wedge L] \qquad \P \hyphen \as
\end{align*}
Since $\E[\omega(0)]=\infty$, we get the desired conclusion by letting $L \to \infty$. 
Therefore, the quenched Lyapunov exponent does not exist in the sense of Proposition~\ref{prop:lyaps}, and we cannot also define the quenched rate function by using the quenched Lyapunov exponent.
Consequently, if $F$ and $G$ are distribution functions on $[0,\infty)$ and one of them does not satisfy (Qu), then the comparisons for $\alpha_F$, $\alpha_G$, $I_F$ and $I_G$ are not well-defined.
As long as we argue comparisons for the Lyapunov exponent and the rate function in the present setting, assumption~(Qu) is necessary in the one-dimensional quenched situation.

On the other hand, in spite of the establishment of (Qu), the annealed Lyapunov exponent is always well-defined in $d=1$.
Therefore, we can expect that the one-dimensional annealed situation is different from the quenched one.
Actually, the next theorem exhibits criteria of the strict comparison for the one-dimensional annealed Lyapunov exponent.

\begin{prop}\label{prop:criteria}
Let $d=1$.
Then, the following results hold:
\begin{enumerate}
	\item\label{item:criteria_equiv}
		If $F \leq G$ and \eqref{eq:add_a} fails to hold (i.e., $F(0) \geq e^{-\beta_G(1)}$), then
		\begin{align*}
			-\log F(0)=\beta_F(1)=\beta_G(1)=-\log G(0).
		\end{align*}
		In particular, if $F$ strictly dominates $G$, then $F(0)<e^{-\beta_G(1)}$ is a necessary and sufficient condition
		for $\beta_F(1)>\beta_G(1)$.
	\item\label{item:criteria_gap}
		If $F$ strictly dominates $G$ and $F(0)<G(0)$, then $\beta_F(1)>\beta_G(1)$ holds.
	\item\label{item:criteria_zero}
		If $F$ strictly dominates $G$ and $F(0)=0$, then $\beta_F(1)>\beta_G(1)$ holds.
\end{enumerate}
\end{prop}
\begin{proof}
We first prove part~\eqref{item:criteria_equiv}.
Assume that $F \leq G$ and $F(0) \geq e^{-\beta_G(1)}$.
For an arbitrary $\epsilon>0$,
\begin{align*}
	\E[e(0,n,\omega_F)]
	&\geq E^0\Bigl[ F(0)^{\#\{z \in \Z:\ell_z(H(n)) \geq 1 \}} \Bigr]\\
	&\geq E^0\Bigl[ F(0)^{\#\{z \in \Z:\ell_z(H(n)) \geq 1 \}} \1{\{ H(n)<H(-\lceil \epsilon n \rceil) \}} \Bigr]\\
	&\geq F(0)^{(1+\epsilon)n} P^0(H(n)<H(-\lceil \epsilon n \rceil)).
\end{align*}
An easy computation shows that
\begin{align*}
	P^0(H(n)<H(-\lceil \epsilon n \rceil))
	= \frac{\lceil \epsilon n \rceil}{n+\lceil \epsilon n \rceil}
\end{align*}
(see for instance \cite[(1.20)]{Law91_book}), and we have
\begin{align*}
	\frac{1}{n}b_F(0,n) \leq -(1+\epsilon)\log F(0)-\frac{1}{n}\log\frac{\lceil \epsilon n \rceil}{n+\lceil \epsilon n \rceil}.
\end{align*}
Hence, letting $n \to \infty$ proves $\beta_F(1) \leq -(1+\epsilon)\log F(0)$.
Since $\epsilon$ is arbitrary, one has
\begin{align}\label{eq:key}
	\beta_F(1) \leq -\log F(0).
\end{align}
This, combined with the assumption $F(0) \geq e^{-\beta_G(1)}$ and the fact that $\beta_F \geq \beta_G$, proves that
\begin{align*}
	F(0) \geq e^{-\beta_G(1)} \geq e^{-\beta_F(1)} \geq F(0),
\end{align*}
which implies $\beta_F(1)=\beta_G(1)=-\log F(0)$.
Furthermore, since \eqref{eq:key} with $F$ replaced by $G$ is valid, we have
\begin{align*}
	G(0) \geq F(0) \geq e^{-\beta_G(1)} \geq G(0),
\end{align*}
and $\beta_G(1)=-\log G(0)$ holds.
With these observations, the first assertion of \eqref{item:criteria_equiv} follows.
For the second assertion of part~\eqref{item:criteria_equiv}, assume that $F$ strictly dominates $G$.
If $F(0)<e^{-\beta_G(1)}$ holds, then $\beta_F(1)>\beta_G(1)$ follows form Theorem~\ref{thm:strict_alyap}.
Conversely, suppose that $\beta_F(1)>\beta_G(1)$ holds.
Then, the first assertion of part~\eqref{item:criteria_equiv} implies $F(0)<e^{-\beta_G(1)}$.
Therefore, the second assertion of part~\eqref{item:criteria_equiv} is proved.

We next prove part~\eqref{item:criteria_gap}.
Note that $-\log F(0)>-\log G(0)$ holds if $F(0)<G(0)$.
Hence, the first assertion of Proposition~\ref{prop:criteria}-\eqref{item:criteria_equiv} implies $F(0)<e^{-\beta_G(1)}$.
Therefore, Theorem~\ref{thm:strict_alyap} gives $\beta_F(1)>\beta_G(1)$, and part~\eqref{item:criteria_gap} follows.

Finally, part~\eqref{item:criteria_zero} is a direct consequence of Theorem~\ref{thm:strict_alyap}.
Indeed, since $\beta_G(1)$ is finite, $F(0)<e^{-\beta_G(1)}$ holds if $F(0)=0$.
Hence, Theorem~\ref{thm:strict_alyap} leads to $\beta_F(1)>\beta_G(1)$.
\end{proof}

Proposition~\ref{prop:criteria} guarantees the existence of the threshold for the coincidence of one-dimensional annealed rate functions as follows.

\begin{cor}\label{cor:criteria_arate}
Let $d=1$.
Suppose that $F$ strictly dominates $G$ and \eqref{eq:add_a} fails to hold (i.e., $F(0) \geq e^{-\beta_G(1)}$).
Then, there exists a constant $v_0 \in (0,1)$ (which may depend on $F$ and $G$) such that
\begin{align}\label{eq:v}
	J_F(x)-J_G(x)
	\begin{cases}
		>0, & \text{if } v_0<|x|<1,\\
		=0, & \text{if } |x| \leq v_0.		
	\end{cases}
\end{align}
\end{cor}
\begin{proof}
Our proof starts with the observation that for any $x \in \R$ with $0<|x|<1$,
\begin{align}\label{eq:J_eqiv}
	J_F(x)-J_G(x)>0 \,\Longleftrightarrow\,
	\lambda_F^\mathrm{an}(x)>0 \text{ or } \lambda_G^\mathrm{an}(x)>0.
\end{align}
To this end, fix $x \in \R$ with $0<|x|<1$.
Note that $\lambda_G^\mathrm{an}(x)$ is finite as seen in the proof of Corollary~\ref{cor:strict_rate}.
We first treat the case where $\lambda_G^\mathrm{an}(x)>0$.
Let $\tilde{F}$ and $\tilde{G}$ be the distribution functions of $\omega_F(0)+\lambda_G^\mathrm{an}(x)$ and $\omega_G(0)+\lambda_G^\mathrm{an}(x)$, respectively.
Then, we have $\tilde{F}(0)=0$, and Proposition~\ref{prop:criteria}-\eqref{item:criteria_zero} implies that
\begin{align*}
	J_F(x)-J_G(x)
	&\geq \beta_F(\lambda_G^\mathrm{an}(x),x)-\beta_G(\lambda_G^\mathrm{an}(x),x)\\
	&= \beta_{\tilde{F}}(x)-\beta_{\tilde{G}}(x)>0.
\end{align*}
Next, in the case where $\lambda_G^\mathrm{an}(x)=0$ but $\lambda_F^\mathrm{an}(x)>0$, there exists $\lambda'>0$ such that $\partial_- \beta_F(\lambda',x)>1$.
Then, since $\beta_F \geq \beta_G$ and $\beta_F(\lambda,x)$ is concave in $\lambda$, one has
\begin{align*}
	J_F(x)-J_G(x)
	&\geq \biggl( \frac{\beta_F(\lambda',x)-\beta_F(0,x)}{\lambda'}-1 \biggr) \lambda'\\
	&\geq (\partial_- \beta_F(\lambda',x)-1)\lambda'>0.
\end{align*}
Finally consider the case where $\lambda_F^\mathrm{an}(x)=\lambda_G^\mathrm{an}(x)=0$.
Then, Proposition~\ref{prop:criteria}-\eqref{item:criteria_equiv} gives
\begin{align*}
	J_F(x)-J_G(x)=\beta_F(x)-\beta_G(x)=0.
\end{align*}
With these observations, \eqref{eq:J_eqiv} immediately follows.

We now refer to the following result obtained by Kosygina--Mountford~\cite[Theorem~1.2]{KosMou12} in our setting:
If $\phi$ is a distribution function on $[0,\infty)$ satisfying that $\P(\omega_\phi(0)=0)<1$ and $\essinf \omega_\phi(0)=0$, then there exists a constant $v_\phi \in (0,1)$ such that
\begin{align}\label{eq:KM}
	\partial_+\beta_\phi(0,1)=\frac{1}{v_\phi},
\end{align}
where $\partial_+\beta_\phi(\lambda,x)$ stands for the right derivative of $\beta_\phi(\lambda,x)$ with respect to $\lambda$.
Then, we prove that if $\phi$ is a distribution function on $[0,\infty)$ satisfying that $\P(\omega_\phi(0)=0)<1$ and $\essinf \omega_\phi(0)=0$, then for $x \in \R$,
\begin{align}\label{eq:anlam}
	\lambda_\phi^\mathrm{an}(x)
	\begin{cases}
		>0, & \text{if } v_\phi<|x|<1,\\
		=0, & \text{if } |x|<v_\phi.
	\end{cases}
\end{align}
In the case where $v_\phi<|x|<1$, \eqref{eq:KM} implies that
\begin{align*}
	\partial_+ \beta_\phi(0,x)
	= |x| \times \partial_+ \beta_\phi(0,1)
	= \frac{|x|}{v_\phi}>1.
\end{align*}
This means that there exists $\lambda_1>0$ such that
\begin{align*}
	\frac{\beta_\phi(\lambda_1,x)-\beta_\phi(0,x)}{\lambda_1}>1.
\end{align*}
By the continuity of $\beta_\phi(\lambda,x)$ in $\lambda$, we can take $\lambda_2 \in (0,\lambda_1)$ such that
\begin{align*}
	\frac{\beta_\phi(\lambda_1,x)-\beta_\phi(\lambda_2,x)}{\lambda_1-\lambda_2}>1.
\end{align*}
This, together with the concavity of $\beta_\phi(\lambda,x)$ in $\lambda$, proves
\begin{align*}
	1<\frac{\beta_\phi(\lambda_1,x)-\beta_\phi(\lambda_2,x)}{\lambda_1-\lambda_2}
	\leq \partial_-\beta_\phi(\lambda_2,x).
\end{align*}
Therefore, $\lambda_\phi^\mathrm{an}(x) \geq \lambda_2>0$ holds in the case where $v_\phi<|x|<1$.
On the other hand, if $|x|<v_\phi$ holds, then
\begin{align*}
	\partial_+\beta_\phi(0,x)	=\frac{|x|}{v_\phi}<1,
\end{align*}
and we can easily see that $\lambda_\phi^\mathrm{an}(x)=0$.
Therefore, \eqref{eq:anlam} is proved.

We are now in a position to prove Corollary~\ref{cor:criteria_arate}.
Note that since $F$ strictly dominates $G$, Lemma~\ref{lem:pseudo}-\eqref{item:pseudo_0} implies $\P(\omega_F(0)=0)<1$.
In addition, since \eqref{eq:add_a} fails to hold, we have
\begin{align*}
	G(0) \geq F(0) \geq e^{-\beta_G(1)}>0,
\end{align*}
which proves that $\essinf \omega_F(0)=\essinf \omega_G(0)=0$.
Hence, \eqref{eq:anlam} holds for $F$.
Assume that $\P(\omega_G(0)=0)<1$.
Then, \eqref{eq:anlam} is also established for $G$.
It follows that $\lambda_F^\mathrm{an}(x)>0$ or $\lambda_G^\mathrm{an}(x)>0$ holds for $v_F \wedge v_G<|x|<1$, and $\lambda_F^\mathrm{an}(x)=\lambda_G^\mathrm{an}(x)=0$ holds for $|x|<v_F \wedge v_G$.
Therefore, \eqref{eq:J_eqiv} shows that
\begin{align*}
	J_F(x)-J_G(x)
	\begin{cases}
		>0, & \text{if } v_F \wedge v_G<|x|<1,\\
		=0, & \text{if } |x|<v_F \wedge v_G.
	\end{cases}
\end{align*}
This is also valid for $|x|=v_F \wedge v_G$ because $J_F$ and $J_G$ are continuous on $[-1,1]$ (see the statement stated above Proposition~\ref{prop:ldp}).
Thus, in the case where $\P(\omega_G(0)=0)<1$, \eqref{eq:v} follows by taking $v_0:=v_F \wedge v_G$.
For the case where $\P(\omega_G(0)=0)=1$, taking $v_0:=v_F$ establishes \eqref{eq:v}.
Indeed, Proposition~\ref{prop:criteria}-\eqref{item:criteria_equiv} gives that for all $x \in \R$,
\begin{align*}
	\beta_G(x)=-\log G(0)=0,
\end{align*}
which implies that $\lambda_G^\mathrm{an}(x)=0$ holds for all $x \in \R$.
This together with \eqref{eq:J_eqiv} and \eqref{eq:anlam} tells us that \eqref{eq:v} holds for $v_0=v_F$.
Consequently, we can find the desired constant $v_0$ in any case, and the proof is complete.
\end{proof}

As seen above, \eqref{eq:add_a} is an important condition for strict comparisons for the one-dimensional annealed Lyapunov exponent and rate function.
Unfortunately, in the one-dimensional case, we did not know whether or not the simple random walk in random potentials always satisfies \eqref{eq:add_a}.
This problem is of interest as future work.



\end{document}